\documentclass[british,english]{article}
\usepackage[T1]{fontenc}
\usepackage[latin9]{inputenc}
\usepackage{amsthm}
\usepackage{amsmath}
\usepackage{amssymb}
\usepackage{esint}

\makeatletter
\numberwithin{equation}{section}
  \theoremstyle{plain}
  \newtheorem{thm}{\protect\theoremname}[section]
  \theoremstyle{remark}
  \newtheorem{rem}{\protect\remarkname}[section]
  \theoremstyle{definition}
  \newtheorem{defn}{\protect\definitionname}[section]
  \theoremstyle{plain}
  \newtheorem{prop}{\protect\propositionname}[section]
  \theoremstyle{plain}
  \newtheorem{lem}{\protect\lemmaname}[section]
  \theoremstyle{plain}
  \newtheorem{cor}{\protect\corollaryname}[section]

\makeatother

\usepackage{babel}
  \addto\captionsbritish{\renewcommand{\definitionname}{Definition}}
  \addto\captionsbritish{\renewcommand{\lemmaname}{Lemma}}
  \addto\captionsbritish{\renewcommand{\propositionname}{Proposition}}
  \addto\captionsbritish{\renewcommand{\remarkname}{Remark}}
  \addto\captionsenglish{\renewcommand{\definitionname}{Definition}}
  \addto\captionsenglish{\renewcommand{\lemmaname}{Lemma}}
  \addto\captionsenglish{\renewcommand{\propositionname}{Proposition}}
  \addto\captionsenglish{\renewcommand{\remarkname}{Remark}}
  \providecommand{\definitionname}{Definition}
  \providecommand{\lemmaname}{Lemma}
  \providecommand{\propositionname}{Proposition}
  \providecommand{\remarkname}{Remark}
\addto\captionsbritish{\renewcommand{\corollaryname}{Corollary}}
\addto\captionsbritish{\renewcommand{\theoremname}{Theorem}}
\addto\captionsenglish{\renewcommand{\corollaryname}{Corollary}}
\addto\captionsenglish{\renewcommand{\theoremname}{Theorem}}
\providecommand{\corollaryname}{Corollary}
\providecommand{\theoremname}{Theorem}

\begin{document}

\title{Quasi-sure Existence of Gaussian Rough Paths and Large Deviation
Principles for Capacities}

\author{H. Boedihardjo%
\thanks{Oxford-Man Institute, University of Oxford, Oxford OX2 6ED, England.
Email: horatio.boedihardjo@oxford-man.ox.ac.uk %
}, X. Geng%
\thanks{Mathematical Institute, University of Oxford, Oxford OX2 6GG and the
Oxford-Man Institute, Eagle House, Walton Well Road, Oxford OX2 6ED.\protect \\
 Email: xi.geng@maths.ox.ac.uk %
} and Z. Qian%
\thanks{Exeter College, University of Oxford, Oxford OX1 3DP, England. Email:
qianz@maths.ox.ac.uk.%
}}
\date{}
\maketitle
\begin{abstract}
We construct a quasi-sure version (in the sense of Malliavin) of geometric
rough paths associated with a Gaussian process with long-time memory.
As an application we establish a large deviation principle (LDP) for
capacities for such Gaussian rough paths. Together with Lyons' universal
limit theorem, our results yield immediately the corresponding results
for pathwise solutions to stochastic differential equations driven
by such Gaussian process in the sense of rough paths. Moreover, our
LDP result implies the result of Yoshida on the LDP for capacities
over the abstract Wiener space associated with such Gaussian process.
\end{abstract}

\section{Introduction}

The theory of rough paths, established by Lyons in his groundbreaking
paper \cite{Lyons}, gives us a fundamental way of understanding path
integrals along one forms and pathwise solutions to differential equations
driven by rough signals. After his work, the study of the (geometric)
rough path nature of stochastic processes (e.g. Brownian motion, Markov
processes, martingales, Gaussian processes, etc.) becomes rather important,
since it will then immediately lead to a pathwise theory of stochastic
differential equations driven by such processes, which is one of the
central problems in stochastic analysis. The rough path regularity
of Brownian motion was first studied in the unpublished Ph.D. thesis
of Sipiläinen \cite{Sipilainen}. Later on Coutin and Qian \cite{CQ}
proved that the sample paths of fractional Brownian motion with Hurst
parameter $H>1/4$ can be lifted as geometric rough paths in a canonical
way, and such canonical lifting does not exist when $H\leqslant1/4$.
Of course their result covers the Brownian motion case. The systematic
study of stochastic processes as rough paths then appeared in the
monographs on rough path theory by Lyons and Qian \cite{LQ} and by
Friz and Victoir \cite{FV}.

The continuity of the solution map for rough differential equations,
which was also proved by Lyons \cite{Lyons} and usually known as
the universal limit theorem, is a fundamental result in rough path
theory. To some extend it gives us a way of understanding the right
topology under which differential equations are stable on rough path
space. An easy but important application of the universal limit theorem
is large deviation principles (or simply LDPs) for pathwise solutions
to stochastic differential equations according to the contraction
principle, once the LDP for the law of the driving process as rough
paths is established under the rough path topology. This is also the
main motivation of strengthening the classical LDPs for probability
measures on path space under the uniform topology to the rough path
setting. Since the rough path topology is stronger than the uniform
topology, a direct corollary is the classical Freidlin-Wentzell theory
on path space, which does not follow immediately from the contraction
principle and is in fact highly nontrivial as the solution map is
not continuous in this case. In the case of Brownian motion, Ledoux,
Qian and Zhang first established the LDP for the law of Brownian rough
paths. Their result was then extended to the case of fractional Brownian
motion by Millet and Sanz-Solé \cite{MS}. The general study of LDPs
for different stochastic processes in particular for Gaussian processes
as rough paths can be found in \cite{FV}. 

We first recall some basic notions from rough path theory which we
use throughout the rest of this article. We refer the readers to \cite{FV},
\cite{LCL}, \cite{LQ} for a detailed presentation.

For $n\geqslant1,$ let 
\[
T^{(n)}\left(\mathbb{R}^{d}\right)=\oplus_{i=0}^{n}\left(\mathbb{R}^{d}\right)^{\otimes i}
\]
be the truncated tensor algebra over $\mathbb{R}^{d}$ degree $n$,
where $\left(\mathbb{R}^{d}\right)^{\otimes0}:=0.$ We use $\Delta$
to denote the standard $2$-simplex $\{(s,t):\ 0\leqslant s\leqslant t\leqslant1\}$.

We call an $\mathbb{R}^{d}$-valued continuous paths over $[0,1]$
\textit{smooth} if it has bounded total variation. Given a smooth
path $w,$ for $k\in\mathbb{N}$ define

\begin{equation}
w_{s,t}^{k}=\int\limits _{s<t_{1}<\cdots<t_{k}<t}dw_{t_{1}}\otimes\cdots\otimes dw_{t_{k}},\ (s,t)\in\Delta.\label{signature}
\end{equation}
From classical integration theory we know that (\ref{signature})
is well-defined as the limit of Riemann-Stieltjes sums. Let $\boldsymbol{w}:\Delta\rightarrow T^{(n)}\left(\mathbb{R}^{d}\right)$
be the functional given by 
\[
\boldsymbol{w}_{s,t}=\left(1,w_{s,t}^{1},\cdots,w_{s,t}^{n}\right)\text{, \ \ \ }(s,t)\in\Delta\text{.}
\]
This is usually called the \textit{lifting} of $w$ up to degree $n$.
The additivity property of integration over disjoint intervals is
then summarized as the following so-called \textit{Chen's identity}:
\begin{equation}
\boldsymbol{w}_{s,u}\otimes\boldsymbol{w}_{u,t}=\boldsymbol{w}_{s,t}\text{ \ \ \ }\forall0\leqslant s\leqslant u\leqslant t\leqslant1.\label{Chen's identity}
\end{equation}
We use $\Omega_{n}^{\infty}\left(\mathbb{R}^{d}\right)$ to denote
the space of all such functionals which are liftings of smooth paths
$w.$ In the definition of $\Omega_{n}^{\infty}$, the starting point
of the path is irrelevant, and we always assume that paths start at
the origin.

Let $p\geqslant1$ be fixed and $[p]$ denote the integer part of
$p$ (not greater than $p$). The $p$-\textit{variation metric }$d_{p}$
on $\Omega_{[p]}^{\infty}$ is defined by 
\[
d_{p}(\boldsymbol{u},\boldsymbol{w})=\max_{1\leqslant i\leqslant[p]}\sup_{D}\left(\sum_{l}\left\vert u_{t_{l-1},t_{l}}^{i}-w_{t_{l-1},t_{l}}^{i}\right\vert ^{\frac{p}{i}}\right)^{\frac{i}{p}},
\]
where the supremum $\sup_{D}$ is taken over all possible finite partitions
of $[0,1]$. The completion of $\Omega_{[p]}^{\infty}$ under $d_{p}$
is called the space of \textit{geometric $p$-rough paths} over $\mathbb{R}^{d}$,
and it is denoted by $G\Omega_{p}\left(\mathbb{R}^{d}\right)$. If
$\boldsymbol{w}=\left(1,w^{1},\cdots,w^{[p]}\right)\in G\Omega_{p}\left(\mathbb{R}^{d}\right)$,
then $\boldsymbol{w}$ also satisfies Chen's identity (\ref{Chen's identity})
in $T^{[p]}\left(\mathbb{R}^{d}\right)$, and $\boldsymbol{w}$ has
finite $p$-variation in the sense that $\sup_{D}\sum_{l}\left\vert w_{t_{l-1},t_{l}}\right\vert ^{\frac{p}{i}}<\infty$
for all $1\leqslant i\leqslant\lbrack p]$.

The fundamental result in rough path theory is the following so-called
\textit{Lyons' universal limit theorem} (see \cite{Lyons}, and also
\cite{FV}, \cite{LQ}) for differential equations driven by geometric
rough paths.
\begin{thm}
\label{Universal Limit Theorem}Let $\{V_{1},\cdots,V_{d}\}$ be a
family of $\gamma$-Lipschitz vector fields on $\mathbb{R}^{N}$ for
some $\gamma>p.$ For any given $x_{0}\in\mathbb{R}^{N},$ define
the map
\[
F(x_{0},\cdot):\ \Omega_{[p]}^{\infty}\left(\mathbb{R}^{d}\right)\rightarrow G\Omega_{p}\left(\mathbb{R}^{N}\right)
\]
in the following way. For any $\boldsymbol{w}\in\Omega_{[p]}^{\infty}\left(\mathbb{R}^{d}\right)$
which is the lifting of some smooth path $w$, let $x$ be the unique
smooth path which is the solution in $\mathbb{R}^{N}$ of the ODE
\[
dx_{t}=\sum_{\alpha=1}^{d}V_{\alpha}(x_{t})dw_{t}^{\alpha},\ t\in[0,1],
\]
with initial value $x_{0}.$ $F(x_{0},\boldsymbol{w})$ is then defined
to be the lifting of $x$ in $\Omega_{p}^{\infty}(\mathbb{R}^{N}).$
Then the map $F(x_{0},\cdot)$ is uniformly continuous on bounded
sets with respect to the $p$-variation metric. 
\end{thm}

\begin{rem}
Theorem \ref{Universal Limit Theorem} is not the original version
of Lyons' result in \cite{Lyons} but an equivalent form. The original
result of Lyons is formulated in terms of rough path integrals and
does not restrict to geometric rough paths only. Here we state the
result in a more elementary form to avoid the machinery of rough path
integrals. 
\end{rem}

The theory of rough paths can be applied to quasi-sure analysis for
Gaussian measures on path space. The notion of quasi-sure analysis
was originally introduced by Malliavin \cite{Malliavin82} (see also
\cite{Malliavin}) to the study of non-degenerate conditioning and
disintegration of Gaussian measures on abstract Wiener spaces. The
fundamental concept in quasi-sure analysis is capacity, which specifies
more precise scales for ``negligible'' subsets of an abstract Wiener
space. In particular, a set of capacity zero is always a null set,
while in general a null set may have positive capacity. According
to Malliavin, the theory of quasi-sure analysis can be regarded as
an infinite dimensional version of non-linear potential theory. It
enables us to disintegrate a Gaussian measure continuously in the
infinite dimensional setting, which for instance applies to the study
of bridge processes and pinned diffusions. Moreover, it also leads
to sharper estimates than classical methods.

The main goal of the present article is to initiate the study of Gaussian
rough paths in the setting of quasi-sure analysis. Due to power tools
in rough path theory, our results lead to the verification of many
classical results for the quasi-sure analysis on Wiener space.

The first aim of this article is to study the quasi-sure existence
of canonical lifting for sample paths of Gaussian processes as geometric
rough paths. The Brownian motion case was studied by Inahama \cite{Inahama1}
under the $p$-variation metric, and Aida \cite{Aida}, Higuchi \cite{Higuchi},
Inahama \cite{Inahama2} and Watanabe \cite{Watanabe} independently
under the Besov norm, by exploiting methods from the Malliavin calculus.
More precisely, it was proved that for quasi-surely, Brownian sample
paths can be lifted as geometric $p$-rough paths for $2<p<3$ . In
the next section, we extend this result to a class of Gaussian processes
with long-time memory which includes fractional Brownian motion, by
applying techniques both from rough path theory and the Malliavin
calculus. Combining our result with Lyons' universal limit theorem,
we obtain immediately a quasi-sure limit theorem for pathwise solutions
to stochastic differential equations driven by Gaussian processes,
which improves the Wong-Zakai type limit theorem and its quasi-sure
version (see for example Ren \cite{Ren}, Malliavin-Nualart \cite{MN}
and the references therein).

The technique we use in the next section enables us to establish a
large deviation principle for capacities for Gaussian rough paths
with long-time memory, which is the second aim of this article. LDPs
for capacities for transformations on an abstract Wiener space was
first studied by Yoshida \cite{Yoshida}. The general definition and
the basic properties of LDPs for induced capacities on a Polish space
first appeared in Gao and Ren \cite{GR}, in which the case of stochastic
flows driven by Brownian motion was also investigated. Before establishing
our LDP result, we first prove two fundamental results on transformations
of LDPs for capacities: the contraction principle and exponential
good approximations, which are both easy adaptations from the classical
results for probability measures. Our LDP result is then based on
the result and method developed in the next section and finite dimensional
approximations. It turns out that the general result of Yoshida in
the case of Gaussian processes is a direct corollary of our result
due to the continuity of the projection map from a geometric rough
path onto its first level path. The original proof of Yoshida relies
crucially on the infinite dimensional structure of abstract Wiener
space, and in particular deep properties of capacity and analytic
properties of the Ornstein-Uhlenbeck semigroup. However, our technique
here replies only on basic properties of capacity and finite dimensional
Gaussian spaces. Moreover, again from Lyons' universal limit theorem,
our LDP result immediately yields the LDPs for capacities for pathwise
solutions to stochastic differential equations driven by Gaussian
processes. In this respect our result is stronger than the result
of Yoshida since we are working in a stronger topology (the $p$-variation
topology) instead of the uniform topology, which is too weak to support
the continuity of the solution map for differential equations. It
is also interesting to note that Inahama \cite{Inahama2} was already
able to applied techniques from quasi-sure analysis to establish LDPs
for pinned diffusion measures.

\section{Quasi-sure Existence of Gaussian Rough Paths}

In the present article, we consider the following class of Gaussian
processes with long-time memory in the sense of Coutin-Qian \cite{CQ}.
\begin{defn}
\label{long time memory}A $d$-dimensional centered, continuous Gaussian
process $\{B_{t}\}_{t\geqslant0}$ starting at the origin with independent
components is said to have $h$-\textit{long-time memory} for some
$0<h<1$ and if there is a constant $C_{h}$ such that 
\[
\mathbb{E}\left[\left\vert B_{t}-B_{s}\right\vert ^{2}\right]\leqslant C_{h}\left\vert t-s\right\vert ^{2h}
\]
for $s,t\geqslant0$ and
\[
\left\vert \mathbb{E}\left[\left(B_{t}^{i}-B_{s}^{i}\right)\left(B_{t+\tau}^{i}-B_{s+\tau}^{i}\right)\right]\right\vert \leqslant C_{h}\tau^{2h}\left\vert \frac{t-s}{\tau}\right\vert ^{2}
\]
for $1\leqslant i\leqslant d,$ $s,t\geqslant0$, $\tau>0$ with $(t-s)/\tau\leqslant1$. 
\end{defn}

A fundamental example of Gaussian processes with long-time memory
is fractional Brownian motion with $h$ being the Hurst parameter
(see \cite{LQ}). 

From now on, we always assume that such Gaussian process is realized
on the path space over the finite time period $[0,1]$. This is of
course equivalent to the consideration of the process over any $[0,T].$
Let $W$ be the space of all $\mathbb{R}^{d}$-valued continuous paths
$w$ over $[0,1]$ with $w_{0}=0,$ and equip $W$ with the Borel
$\sigma$-algebra $\mathcal{B}(W)$. Let $\mathbb{P}$ be the law
on $(W,\mathcal{B}(W))$ of some Gaussian process with $h$-long-time
memory in the sense of Definition \ref{long time memory}.

It is a fundamental result of Coutin and Qian \cite{CQ} that if $h>1/4,\ 2<p<4$
with $hp>1,$ then outside a $\mathbb{P}$-null set each sample path
$w\in W$ can be lifted as geometric $p$-rough paths in a canonical
way. More precisely, for $m\geqslant1$, let $t_{m}^{k}=k/2^{m}$
($k=0,1,\cdots,2^{m}$) be the $m$-th dyadic partition of $[0,1].$
Given $w\in W$, define $w^{(m)}$ to be the dyadic piecewise linear
interpolation of $w$ by 
\[
w_{t}^{(m)}=w_{t_{m}^{k-1}}+2^{m}\left(t-t_{m}^{k-1}\right)\left(w_{t_{m}^{k}}-w_{t_{m}^{k-1}}\right),\ t\in\left[t_{m}^{k-1},t_{m}^{k}\right],
\]
and let 
\[
\boldsymbol{w}_{s,t}^{(m)}=\left(1,w_{s,t}^{(m),1},w_{s,t}^{(m),2},w_{s,t}^{(m),3}\right),\ (s,t)\in\Delta,
\]
be the geometric rough path associated with $w^{(m)}$ up to level
3. Let $\mathcal{A}_{p}$ be the totality of all $w\in W$ such that
$\{\boldsymbol{w}^{(m)}\}_{m\geqslant1}$ is a Cauchy sequence under
the $p$-variation metric $d_{p}$. Then $\mathcal{A}_{p}^{c}$ is
a $\mathbb{P}$-null set and hence $\boldsymbol{w}^{(m)}$ converges
to a unique geometric $p$-rough path $\boldsymbol{w}$ for $\mathbb{P}$-almost-surely.
The convergence holds in $L^{1}(W,\mathbb{P})$ as well.
\begin{rem}
Although a geometric $p$-rough path is defined up to level $[p],$
by Lyons' extension theorem (see \cite{Lyons}) it does not make a
difference if we always consider up to level $3$ under $d_{p}$ since
$2<p<4$.
\end{rem}

\begin{rem}
Coutin and Qian \cite{CQ} also showed that if $h\leqslant1/4,$ no
subsequence of $\boldsymbol{w}_{s,t}^{(m)}$ converges in probability
or in $L^{1}$, and hence such canonical lifting of sample paths as
geometric rough paths does not exist.
\end{rem}

The goal of this section is to strengthen the result of Coutin-Qian
to the quasi-sure setting in the sense of Malliavin. The main result
and technique developed in this section are essential to establish
a large deviation principle for capacities as we will see later on.

Throughout the rest of this article, we fix $h\in(1/4,1/2]$, $p\in(2,4)$
with $hp>1$ (the case of $h>1/2$ is trivial from the rough path
point of view), and consider a $d$-dimensional Gaussian process with
$h$-long-time memory. 

We first recall some basic notions about the Malliavin calculus and
quasi-sure analysis. We refer the readers to \cite{Malliavin}, \cite{Nualart}
for a systematic discussion.

Let $\mathcal{H}$ be the Cameron-Martin space associated with the
corresponding Gaussian measure $\mathbb{P}$ on $W$. $\mathcal{H}$
is canonically defined to be the space of all paths in $W$ of the
form 
\[
h_{t}=\mathbb{E}[Zw_{t}],\ t\in[0,1],
\]
where $Z$ is an element of the $L^{2}$ space generated by the process,
and the inner product is given by $\langle h_{1},h_{2}\rangle=\mathbb{E}[Z_{1}Z_{2}].$
It follows that the identity map $\iota$ defines a continuous and
dense embedding from $\mathcal{H}$ into $W$ which makes $\left(W,\mathcal{H},\mathbb{P}\right)$
into an abstract Wiener space in the sense of Gross. Let $\iota^{*}:\ W^{*}\rightarrow\mathcal{H}^{*}\cong\mathcal{H}$
be the dual of $\iota.$ Then the identity map $\mathcal{I}:\ W^{\ast}\hookrightarrow L^{2}(W,\mathbb{P})$
uniquely extends to an isometric embedding from $\mathcal{H}$ into
$L^{2}(W,\mathbb{P})$ via $\iota^{*}$. 

If $f$ is a smooth Schwarz function on $\mathbb{R}^{n}$, and $\varphi_{1},\cdots,\varphi_{n}\in W^{*}$,
then $F=f(\varphi_{1},\cdots,\varphi_{n})$ is called a \textit{smooth
(Wiener) functional} on $W$. The collection of all smooth functionals
on $W$ is denoted by $\mathcal{S}$. The \textit{Malliavin derivative}
of $F$ is defined to be the $\mathcal{H}$-valued functional 
\[
DF=\sum_{i=1}^{n}\frac{\partial f}{\partial x^{i}}(\varphi_{1},\cdots,\varphi_{n})\iota^{*}\varphi_{i}\text{,}
\]
Such definition can be generalized to smooth functionals taking values
in a separable Hilbert space $E$. Let $\mathcal{S}(E)$ be the space
of $E$-valued functionals of the form $F=\sum_{i=1}^{k}F_{i}e_{i}$,
where $F_{i}\in\mathcal{S}$, $e_{i}\in E$. The Malliavin derivative
of $F$ is defined to be the $\mathcal{H}\otimes E$-valued functional
$DF=\sum_{i=1}^{k}DF_{i}\otimes e_{i}$. Such definition is independent
of the form of $F$, and by induction we can define higher order derivatives
$D^{N}F$ for $N\in\mathbb{N}$, which is then an $\mathcal{H}^{\otimes N}\otimes E$-valued
functional. Given $q\geqslant1,\ N\in\mathbb{N}$, the Sobolev norm
$\Vert\cdot\Vert_{q,N;E}$ on $\mathcal{S}(E)$ is defined by 
\[
\Vert F\Vert_{q,N;E}=\left(\sum_{i=0}^{N}\mathbb{E}\left[\left\Vert D^{i}F\right\Vert _{\mathcal{H}^{\otimes i}\otimes E}^{q}\right]\right)^{\frac{1}{q}}.
\]
The completion of $(\mathcal{S}(E),\Vert\cdot\Vert_{q,N;E})$ is called
the $(q,N)$-\textit{Sobolev space} for $E$-valued functionals over
$W$, and it is denoted by $\mathbb{D}_{N}^{q}(E)$.

For any $q>1,N\in\mathbb{N},$ the $(q,N)$-\textit{capacity} Cap$_{q,N}$
is a functional defined on the collection of all subsets of $W$.
If $O$ is an open subset of $W$, then 
\[
\text{Cap}_{q,N}(O):=\inf\left\{ \Vert u\Vert_{q,N}:\ u\in\mathbb{D}_{N}^{q},\ u\geqslant1\text{ on }O,\ u\geqslant0\ \text{on }W\text{,}\ \mathbb{P}\mbox{-a.s.}\right\} 
\]
and for any arbitrary subset $A$ of $W$, 
\[
\text{Cap}_{q,N}(A):=\inf\left\{ \text{Cap}_{q,N}(O):\ O\ \text{open, }A\subset O\right\} .
\]
A subset $A\subset W$ is called \textit{slim} if Cap$_{q,N}(A)=0$
for all $q>1$ and $N\in\mathbb{N}$. A property for paths in $W$
is said to hold for \textit{quasi-surely if} it holds outside a slim
set.

The $(q,N)$-capacity has the following basic properties:

(1) if $A\subset B$, then 
\[
0\leqslant\mbox{Cap}_{q,N}(A)\leqslant\mbox{Cap}_{q,N}(B);
\]

(2) Cap$_{q,N}$ is increasing in $q$ and $N$;

(3) Cap$_{q,N}$ is sub-additive, i.e.,
\[
\mbox{Cap}_{q,N}\left(\bigcup_{i=1}^{\infty}A_{i}\right)\leqslant\sum_{i=1}^{\infty}\mbox{Cap}_{q,N}(A_{i}).
\]

The following quasi-sure version of Tchebycheff's inequality and Borel-Cantelli's
lemma play an essential role in the study of quasi-sure convergence
in our approach. We refer the readers to \cite{Malliavin} for the
proof.
\begin{prop}
\label{Tchebycheff and Borel-Cantelli}(1) For any $\lambda>0$ and
any $u\in\mathbb{D}_{N}^{q}$ which is lower semi-continuous, we have
\[
\mathrm{Cap}_{q,N}\left\{ w\in W:u(w)>\lambda\right\} \leqslant\frac{C_{q,N}}{\lambda}\Vert u\Vert_{q,N},
\]
where $C_{q,N}$ is a constant depending only on $q$ and $N$.

(2) For any sequence $\{A_{n}\}_{n=1}^{\infty}$ of subsets of $W,$
if $\sum_{n=1}^{\infty}\mathrm{Cap}_{q,N}(A_{n})<\infty,$ then 
\[
\mathrm{Cap}_{q,N}\left(\limsup_{n\rightarrow\infty}A_{n}\right)=0.
\]

\end{prop}

Now we are in a position to state our main result of this section.
\begin{thm}
\label{quasi-sure convergence}Suppose that $\mathbb{P}$ is the Gaussian
measure on $(W,\mathcal{B}(W))$ associated with a $d$-dimensional
Gaussian process with $h$-long-time memory for some $h\in(1/4,1/2],\ p\in(2,4)$
with $hp>1.$ Then $\mathcal{A}_{p}^{c}$ is a slim set. In particular,
sample paths of the Gaussian processes can be lifted as geometric
$p$-rough paths in a canonical way quasi-surely, as the limit of
the lifting of dyadic piecewise linear interpolation under $d_{p}.$
\end{thm}

By applying Lyons' universal limit theorem (Theorem \ref{Universal Limit Theorem})
for rough differential equations driven by geometric rough paths,
an immediate consequence of Theorem \ref{quasi-sure convergence}
is the quasi-sure existence and uniqueness for pathwise solutions
to stochastic differential equations driven by Gaussian processes
with $h$-long-time memory in the sense of geometric rough paths,
under certain regularity conditions on the generating vector fields. 

The main idea of proving Theorem \ref{quasi-sure convergence} is
to use a crucial control on the $p$-variation metric which is defined
over dyadic partitions only, and to apply basic results for Gaussian
polynomials in the Malliavin calculus.

If $\boldsymbol{w}=(1,w^{1},w^{2},w^{3})$ and $\boldsymbol{\tilde{w}}=(1,\tilde{w}^{1},\tilde{w}^{2},\tilde{w}^{3})$
are two functionals on $\Delta$ taking values in $T^{3}(\mathbb{R}^{d})$,
define 
\begin{equation}
\rho_{i}(\boldsymbol{w},\boldsymbol{\tilde{w}})=\left(\sum_{n=1}^{\infty}n^{\gamma}\sum_{k=1}^{2^{n}}\left\vert w_{t_{n}^{k-1},t_{n}^{k}}^{i}-\tilde{w}_{t_{n}^{k-1},t_{n}^{k}}^{i}\right\vert ^{\frac{p}{i}}\right)^{\frac{i}{p}},\label{the rho function}
\end{equation}
where $i=1,2,3$ and $\gamma>p-1$ is a fixed constant . We  use $\rho_{j}(\boldsymbol{w})$
to denote $\rho_{j}(\boldsymbol{w},\boldsymbol{\tilde{w}})$ with
$\boldsymbol{\tilde{w}}=(1,0,0,0)$. These functionals were originally
introduced by Hambly and Lyons \cite{HL} for constructing the stochastic
area processes associated with Brownian motions on the Sierpinski
gasket. They were then used by Ledoux, Qian and Zhang \cite{LQZ}
to establish a large deviation principle for Brownian rough paths
under the $p$-variation topology. We  also use these functionals
to prove a large deviation principle for capacity in the next section. 

The following estimate is contained implicitly in \cite{HL} and made
explicit in \cite{LQ}.
\begin{lem}
\label{control of p-var metric as lemma}There exists a positive constant
$C_{d,p,\gamma}$ depending only on $d,p,\gamma,$ such that for any
$\boldsymbol{w},\widetilde{\boldsymbol{w}},$ 
\begin{eqnarray}
d_{p}(\boldsymbol{w},\widetilde{\boldsymbol{w}}) & \leqslant & C_{d,p,\gamma}\max\left\{ \rho_{1}(\boldsymbol{w},\widetilde{\boldsymbol{w}}),\rho_{2}(\boldsymbol{w},\widetilde{\boldsymbol{w}}),\rho_{1}(\boldsymbol{w},\widetilde{\boldsymbol{w}})\left(\rho_{1}(\boldsymbol{w})+\rho_{1}(\widetilde{\boldsymbol{w}})\right),\right.\nonumber \\
 &  & \rho_{3}(\boldsymbol{w},\widetilde{\boldsymbol{w}}),\rho_{2}(\boldsymbol{w},\widetilde{\boldsymbol{w}})\left(\rho_{1}(\boldsymbol{w})+\rho_{1}(\widetilde{\boldsymbol{w}})\right),\nonumber \\
 &  & \left.\rho_{1}(\boldsymbol{w},\widetilde{\boldsymbol{w}})\left(\rho_{2}(\boldsymbol{w})+\rho_{2}(\widetilde{\boldsymbol{w}})+(\rho_{1}(\boldsymbol{w})+\rho_{1}(\widetilde{\boldsymbol{w}}))^{2}\right)\right\} .\label{control of p-var metric as equality}
\end{eqnarray}

\end{lem}

The main difficulty of proving Theorem \ref{quasi-sure convergence}
is that it is unknown if the $p$-variation metric is a differentiable
in the sense of Malliavin. We  get around this difficulty first by
controlling the $p$-variation metric using Lemma \ref{control of p-var metric as lemma}
and then by observing that the capacity of $\left\{ \rho_{i}\left(\boldsymbol{w}^{(m+1)},\boldsymbol{w}^{(m)}\right)>\lambda\right\} $
is ``evenly distributed'' over the dyadic sub-intervals (see (\ref{evenly distributed})
in the following). Our task is then reduced to the estimation of the
Sobolev norms of certain Gaussian polynomials, which is contained
in the following basic result in the Malliavin calculus (see \cite{Nualart}). 
\begin{lem}
\label{polynomial estimates}Fix $N\in\mathbb{N}$. Let $\mathcal{P}^{N}(E)$
be the space of $E$-valued polynomial functionals of degree less
than or equal to $N$. Then for any $q>2,$ we have 
\begin{equation}
\|F\|_{q;E}\leqslant(N+1)(q-1)^{\frac{N}{2}}\|F\|_{2;E}.\label{equivalence of q-norm and 2-norm}
\end{equation}
Moreover, for any $F\in\mathcal{P}^{N}(E)$ and $i\leqslant N$ we
have 
\begin{equation}
\Vert D^{i}F\Vert_{2;\mathcal{H}^{\otimes i}\otimes E}\leqslant N^{\frac{i+1}{2}}\Vert F\Vert_{2;E}.\label{differential inequality}
\end{equation}

\end{lem}

The following $L^{2}$-estimates for the dyadic piecewise linear interpolation,
which are contained in a series of calculations in \cite{LQ}, are
crucial for us. 
\begin{lem}
\label{L^2 estimates}Let $m,n\geqslant1$ and $k=1,\cdots,2^{n}$. 

1) For $i=1,2,3$, we have 
\[
\left\Vert w_{t_{n}^{k-1},t_{n}^{k}}^{(m),i}\right\Vert _{2;(\mathbb{R}^{d})^{\otimes i}}\leqslant\begin{cases}
C_{d,h}\left(\frac{1}{2^{nh}}\right)^{i}, & n\leqslant m,\\
C_{d,h}\left(\frac{2^{m(1-h)}}{2^{n}}\right)^{i}, & n>m.
\end{cases}
\]

2) We also have 
\begin{eqnarray*}
\left\Vert w_{t_{n}^{k-1},t_{n}^{k}}^{(m+1),1}-w_{t_{n}^{k-1},t_{n}^{k}}^{(m),1}\right\Vert _{2;\mathbb{R}^{d}} & \leqslant & \begin{cases}
0, & n\leqslant m,\\
C_{d,h}\frac{2^{m(1-h)}}{2^{n}}, & n>m;
\end{cases}\\
\left\Vert w_{t_{n}^{k-1},t_{n}^{k}}^{(m+1),2}-w_{t_{n}^{k-1},t_{n}^{k}}^{(m),2}\right\Vert _{2;(\mathbb{R}^{d})^{\otimes2}} & \leqslant & \begin{cases}
C_{d,h}\frac{1}{2^{\frac{1}{2}(4h-1)m}2^{\frac{1}{2}n}}, & n\leqslant m,\\
C_{d,h}\frac{2^{2m(1-h)}}{2^{2n}}, & n>m;
\end{cases}\\
\left\Vert w_{t_{n}^{k-1},t_{n}^{k}}^{(m+1),3}-w_{t_{n}^{k-1},t_{n}^{k}}^{(m),3}\right\Vert _{2;(\mathbb{R}^{d})^{\otimes3}} & \leqslant & \begin{cases}
C_{d,h}\frac{1}{2^{\frac{1}{2}(4h-1)m}2^{\frac{1+2h}{2}n}}, & n\leqslant m,\\
C_{d,h}\frac{2^{3m(1-h)}}{2^{3n}}, & n>m.
\end{cases}
\end{eqnarray*}
Here $C_{d,h}$ is a constant depending only on $d$ and $h.$
\end{lem}

Now we can proceed to the proof of Theorem \ref{quasi-sure convergence}.
The key step is to establish estimates for the capacities of the tail
events $\left\{ w:\ \rho_{i}\left(\boldsymbol{w}^{(m+1)},\boldsymbol{w}^{(m)}\right)>\lambda\right\} $
and $\left\{ w:\ \rho_{i}\left(\boldsymbol{w}^{(m)}\right)>\lambda\right\} $
($i=1,2,3$). This is contained in the following lemma. 
\begin{lem}
\label{estimating tail events}Let $\theta\in\left(\left(\frac{p(2h+1)}{6}-1\right)^{+},hp-1\right),$
$\widetilde{N}>\frac{N}{2}\vee\left(2\left(h-\frac{\theta+1}{p}\right)\right)^{-1}$.
Then we have

(1) 
\begin{equation}
\begin{array}{ccc}
\mathrm{Cap}_{q,N}\left\{ w:\rho_{i}\left(\boldsymbol{w}^{(m+1)},\boldsymbol{w}^{(m)}\right)>\lambda\right\}  & \leqslant & C_{i}\lambda^{-2\widetilde{N}}\left(\frac{1}{2^{m}}\right)^{2i\widetilde{N}\left(h-\frac{\theta+1}{p}\right)-1},\end{array}\label{first estimate}
\end{equation}

(2)
\begin{equation}
\mathrm{Cap}_{q,N}\left\{ w:\rho_{i}\left(\boldsymbol{w}^{(m)}\right)>\lambda\right\} \leqslant C_{i}\lambda^{-2\widetilde{N}}.\label{second estimate}
\end{equation}
Here $C_{i}$ is a positive constant of the form $C_{i}=C_{1}C_{2}^{\tilde{N}}g\left(\tilde{N};N\right)\tilde{N}^{i\tilde{N}},$
where $C_{1}$ depends only on $q$ and $N,$ $C_{2}$ depends only
on $d,p,h,\gamma,\theta,q$ and $g\left(\tilde{N};N\right)$ is a
polynomial in $\tilde{N}$ with degree depending only on $N$ and
universal constant coefficients.\end{lem}
\begin{proof}
For $i=1,2,3$, set
\begin{eqnarray*}
I_{i}(m;\lambda) & = & \mbox{Cap}_{q,N}\left\{ w:\ \rho_{i}\left(\boldsymbol{w}^{\left(m+1\right)},\boldsymbol{w}^{\left(m\right)}\right)>\lambda\right\} \\
 & = & \mbox{Cap}_{q,N}\left\{ w:\ \rho_{i}\left(\boldsymbol{w}^{\left(m+1\right)},\boldsymbol{w}^{\left(m\right)}\right)^{\frac{p}{i}}>\lambda^{\frac{p}{i}}\right\} .
\end{eqnarray*}
By the definition of $\rho_{i},$ for every $\theta>0$ we have
\begin{align*}
 & \left\{ w:\ \rho_{i}\left(\boldsymbol{w}^{\left(m+1\right)},\boldsymbol{w}^{\left(m\right)}\right)^{\frac{p}{i}}>\lambda^{\frac{p}{i}}\right\} \\
\subset & \bigcup\limits _{n=1}^{\infty}\left\{ w:\ \sum_{k=1}^{2^{n}}\left\vert w_{t_{n}^{k-1},t_{n}^{k}}^{(m+1),i}-w_{t_{n}^{k-1},t_{n}^{k}}^{(m),i}\right\vert ^{\frac{p}{i}}>C_{\gamma,\theta}\lambda^{\frac{p}{i}}\left(\frac{1}{2^{n}}\right)^{\theta}\right\} ,
\end{align*}
where $C_{\gamma,\theta}=\left(\sum_{n=1}^{\infty}n^{\gamma}2^{-n\theta}\right)^{-1}$.
Therefore,
\begin{align}
 & \ I_{i}\left(m;\lambda\right)\nonumber \\
\leqslant & \ \sum_{n=1}^{\infty}\mbox{Cap}_{q,N}\left\{ w:\ \sum_{k=1}^{2^{n}}\left\vert w_{t_{n}^{k-1},t_{n}^{k}}^{(m+1),i}-w_{t_{n}^{k-1},t_{n}^{k}}^{(m),i}\right\vert ^{\frac{p}{i}}>\lambda^{\frac{p}{i}}C_{\gamma,\theta}\left(\frac{1}{2^{n}}\right)^{\theta}\right\} \nonumber \\
\leqslant & \ \sum_{n=1}^{\infty}\sum_{k=1}^{2^{n}}\mbox{Cap}_{q,N}\left\{ w:\ \left\vert w_{t_{n}^{k-1},t_{n}^{k}}^{(m+1),i}-w_{t_{n}^{k-1},t_{n}^{k}}^{(m),i}\right\vert ^{\frac{p}{i}}>\lambda^{\frac{p}{i}}C_{\gamma,\theta}\left(\frac{1}{2^{n}}\right)^{\theta+1}\right\} .\label{evenly distributed}
\end{align}

On the other hand, for any $\tilde{N}>0$ we have
\begin{align*}
 & \ \mbox{Cap}_{q,N}\left\{ \left\vert w_{t_{n}^{k-1},t_{n}^{k}}^{(m+1),i}-w_{t_{n}^{k-1},t_{n}^{k}}^{(m),i}\right\vert ^{\frac{p}{i}}>\lambda^{\frac{p}{i}}C_{\gamma,\theta}\left(\frac{1}{2^{n}}\right)^{\theta+1}\right\} \\
= & \ \mbox{Cap}_{q,N}\left\{ f_{m,n,k}^{i}>\left[\lambda C_{\gamma,\theta}^{\frac{i}{p}}\left(\frac{1}{2^{n}}\right)^{\frac{i}{p}(\theta+1)}\right]^{2\tilde{N}}\right\} ,
\end{align*}
where 
\[
\begin{array}{ccc}
f_{m,n,k}^{i}(w) & = & \left\vert w_{t_{n}^{k-1},t_{n}^{k}}^{(m+1),i}-w_{t_{n}^{k-1},t_{n}^{k}}^{(m),i}\right\vert ^{2\tilde{N}},\ \mbox{for }w\in W.\end{array}
\]
Since $\tilde{N}$ is a natural number, $f_{m,n,k}^{i}$ are polynomial
functionals of degree $2i\tilde{N}$, and hence they are $N$ times
differentiable in the sense of Malliavin provided $\tilde{N}\geqslant\frac{N}{2}$.
Consequently, we can apply Tchebycheff's inequality (the first part
of Proposition \ref{Tchebycheff and Borel-Cantelli}) to obtain 
\[
I_{i}(m;\lambda)\leqslant C_{q,N}\sum_{n=1}^{\infty}\sum_{k=1}^{2^{n}}\left(C_{\gamma,\theta}^{\frac{i}{p}}\lambda\left(\frac{1}{2^{n}}\right)^{\frac{i}{p}(\theta+1)}\right)^{-2\widetilde{N}}\left\Vert f_{m,n,k}^{i}\right\Vert _{q,N}.
\]
If $q>2,$ by (\ref{equivalence of q-norm and 2-norm}) of Lemma \ref{polynomial estimates},
we have 
\begin{eqnarray*}
\left\Vert f_{m,n,k}^{i}\right\Vert _{q,N} & \leqslant & \sum_{l=0}^{N}\|D^{l}f_{m,n,k}^{i}\|_{q;\mathcal{H}^{\otimes l}}\\
 & \leqslant & \left(2i\tilde{N}+1\right)(q-1)^{i\tilde{N}}\sum_{l=0}^{N}\left\Vert D^{l}f_{m,n,k}^{i}\right\Vert _{2;\mathcal{H}^{\otimes l}}.
\end{eqnarray*}
By (\ref{differential inequality}) of Lemma \ref{polynomial estimates},
we have 
\[
\left\Vert D^{l}f_{m,n,k}^{i}\right\Vert _{2;\mathcal{H}^{\otimes l}}\leqslant\left(2i\tilde{N}\right)^{\frac{N+1}{2}}\left\Vert f_{m,n,k}^{i}\right\Vert _{2}.
\]
Therefore, 
\[
\left\Vert f_{m,n,k}^{i}\right\Vert _{q,N}\leqslant(N+1)\left(2i\tilde{N}+1\right)(q-1)^{i\tilde{N}}\left(2i\tilde{N}\right)^{\frac{N+1}{2}}\left\Vert f_{m,n,k}^{i}\right\Vert _{2}.
\]
Moreover, since $w_{t_{n}^{k-1},t_{n}^{k}}^{(m+1),i}-w_{t_{n}^{k-1},t_{n}^{k}}^{(m),i}$
is an $(\mathbb{R}^{d})^{\otimes i}$-valued polynomial functional
of degree $i$, we know again from (\ref{equivalence of q-norm and 2-norm})
that 
\begin{eqnarray*}
\left\Vert f_{m,n,k}^{i}\right\Vert _{2} & = & \left\Vert w_{t_{n}^{k-1},t_{n}^{k}}^{(m+1),i}-w_{t_{n}^{k-1},t_{n}^{k}}^{(m),i}\right\Vert _{4\tilde{N};\left(\mathbb{R}^{d}\right)^{\otimes i}}^{2\tilde{N}}\\
 & \leqslant & (i+1)^{2\tilde{N}}\left(4\widetilde{N}-1\right)^{i\tilde{N}}\left\Vert w_{t_{n}^{k-1},t_{n}^{k}}^{(m+1),i}-w_{t_{n}^{k-1},t_{n}^{k}}^{(m),i}\right\Vert _{2;\left(\mathbb{R}^{d}\right)^{\otimes i}}^{2\tilde{N}}.
\end{eqnarray*}
Therefore,
\begin{align}
 & \left\Vert f_{m,n,k}^{i}\right\Vert _{q,N}\nonumber \\
\leqslant & (N+1)\left((q-1)^{i}(i+1)^{2}\right)^{\tilde{N}}\left(2i\tilde{N}+1\right)\left(2i\tilde{N}\right)^{\frac{N+1}{2}}\nonumber \\
 & \cdot\left(4\tilde{N}-1\right)^{i\tilde{N}}\left\Vert w_{t_{n}^{k-1},t_{n}^{k}}^{(m+1),i}-w_{t_{n}^{k-1},t_{n}^{k}}^{(m),i}\right\Vert _{2;\left(\mathbb{R}^{d}\right)^{\otimes i}}^{2\tilde{N}}\nonumber \\
\leqslant & (N+1)\left(1024(q-1)^{3}\right)^{\tilde{N}}\left(6\tilde{N}+1\right)\left(6\tilde{N}\right)^{N}\tilde{N}^{i\tilde{N}}\left\Vert w_{t_{n}^{k-1},t_{n}^{k}}^{(m+1),i}-w_{t_{n}^{k-1},t_{n}^{k}}^{(m),i}\right\Vert _{2;\left(\mathbb{R}^{d}\right)^{\otimes i}}^{2\tilde{N}}.\label{controlling the Sobolev norm}
\end{align}
 Let $C_{i}$ be the constant before $\left\Vert w_{t_{n}^{k-1},t_{n}^{k}}^{(m+1),i}-w_{t_{n}^{k-1},t_{n}^{k}}^{(m),i}\right\Vert _{2;\left(\mathbb{R}^{d}\right)^{\otimes i}}^{2\tilde{N}}$
on the R.H.S. of (\ref{controlling the Sobolev norm}). 

By absorbing the constant in Tchebycheff's inequality into $C_{i}$,
we arrive at
\begin{align}
 & \ I_{i}(m;\lambda)\nonumber \\
\leqslant & \ C_{i}\sum_{n=1}^{\infty}\sum_{k=1}^{2^{n}}\left(C_{\theta}^{\frac{i}{p}}\lambda\left(\frac{1}{2^{n}}\right)^{\frac{i}{p}(\theta+1)}\right)^{-2\widetilde{N}}\left\Vert w_{t_{n}^{k-1},t_{n}^{k}}^{(m+1),i}-w_{t_{n}^{k-1},t_{n}^{k}}^{(m),i}\right\Vert _{2;\left(\mathbb{R}^{d}\right)^{\otimes i}}^{2\tilde{N}}.\label{estimating I_i}
\end{align}
 Exactly the same computation yields
\begin{align}
 & \ \mbox{Cap}_{q,N}\left\{ w:\ \rho_{i}\left(\boldsymbol{w}^{\left(m\right)}\right)>\lambda\right\} \nonumber \\
\leqslant & \ C_{i}\sum_{n=1}^{\infty}\sum_{k=1}^{2^{n}}\left(C_{\gamma,\theta}^{\frac{i}{p}}\lambda\left(\frac{1}{2^{n}}\right)^{\frac{i}{p}(\theta+1)}\right)^{-2\widetilde{N}}\left\Vert w_{t_{n}^{k-1},t_{n}^{k}}^{(m),i}\right\Vert _{2;\left(\mathbb{R}^{d}\right)^{\otimes i}}^{2\tilde{N}}.\label{estimating rho_i}
\end{align}
We now substitute the estimates in Lemma \ref{L^2 estimates} into
(\ref{estimating I_i}) and (\ref{estimating rho_i}). In what follows,
we assume that $\theta\in\left(\left(\frac{p(2h+1)}{6}-1\right)^{+},hp-1\right),$
$\widetilde{N}>\frac{N}{2}\vee\left(2\left(h-\frac{\theta+1}{p}\right)\right)^{-1}$
for summability reason. We  also absorb the constant $C_{d,h}$ in
Lemma \ref{L^2 estimates} and $C_{\gamma,\theta}$.

For $i=1$, this gives 
\begin{align*}
I_{1}\left(m;\lambda\right) & \leqslant C_{1}\lambda^{-2\tilde{N}}2^{2\tilde{N}m\left(1-h\right)}\sum_{n=m+1}^{\infty}\sum_{k=1}^{2^{n}}2^{-2n\tilde{N}\left(1-\frac{\theta+1}{p}\right)}\\
 & \leqslant C_{1}\lambda^{-2\tilde{N}}2^{-m\left(2\tilde{N}\left(h-\frac{\theta+1}{p}\right)-1\right)}.
\end{align*}

For $i=2$, this gives\foreignlanguage{british}{
\begin{eqnarray*}
I_{2}\left(m;\lambda\right) & \leqslant & C_{2}\lambda^{-2\tilde{N}}\left(\sum_{n=1}^{m}\sum_{k=1}^{2^{n}}2^{-n\tilde{N}\left(1-\frac{4(\theta+1)}{p}\right)-m\tilde{N}\left(4h-1\right)}\right.\\
 &  & \left.+\sum_{n=m+1}^{\infty}\sum_{k=1}^{2^{n}}2^{-4n\tilde{N}\left(1-\frac{\theta+1}{p}\right)+4m\tilde{N}\left(1-h\right)}\right)\\
 & \leqslant & C_{2}\lambda^{-2\tilde{N}}2^{-m\left(4\tilde{N}\left(h-\frac{\theta+1}{p}\right)-1\right)}
\end{eqnarray*}
}

For $i=3$, this gives\foreignlanguage{british}{
\begin{eqnarray*}
I_{3}\left(m;\lambda\right) & \leqslant & C_{3}\lambda^{-2\tilde{N}}\left(\sum_{n=1}^{m}\sum_{k=1}^{2^{n}}2^{-n\tilde{N}\left(1+2h-\frac{6\left(\theta+1\right)}{p}\right)-m\tilde{N}\left(4h-1\right)}\right.\\
 &  & \left.+\sum_{n=m+1}^{\infty}\sum_{k=1}^{2^{n}}2^{-6n\tilde{N}\left(1-\frac{\theta+1}{p}\right)+6m\tilde{N}\left(1-h\right)}\right)\\
 & \leqslant & C_{3}\lambda^{-2\tilde{N}}2^{-m\left(6\tilde{N}\left(h-\frac{\theta+1}{p}\right)-1\right)}
\end{eqnarray*}
}

Therefore, for $i=1,2,3,$ we have
\[
I_{i}\left(m;\lambda\right)\leqslant C_{i}\lambda^{-2\tilde{N}}2^{-m\left(2i\tilde{N}\left(h-\frac{\theta+1}{p}\right)-1\right)}
\]
which gives (\ref{first estimate}). From the computation it is easy
to see that the constants $C_{i}$ here are of the form stated in
the lemma.

Similar computation yields that for $i=1,2,3,$ \foreignlanguage{british}{
\begin{eqnarray*}
 &  & \mbox{Cap}_{q,N}\left\{ w:\ \rho_{i}\left(\boldsymbol{w}^{\left(m\right)}\right)>\lambda\right\} \\
 & \leqslant & C_{i}\left(\lambda^{-2\tilde{N}}\sum_{n=1}^{m}\sum_{k=1}^{2^{n}}2^{-2n\tilde{N}i\left(h-\frac{\theta+1}{p}\right)}\right.
\end{eqnarray*}
\begin{align*}
 & \ \left.+\lambda^{-2\tilde{N}}\sum_{n=m+1}^{\infty}\sum_{k=1}^{2^{n}}2^{-2\tilde{N}i\left(n\left(1-\frac{\theta+1}{p}\right)-m\left(1-h\right)\right)}\right)\\
\leqslant & \ C_{i}\lambda^{-2\tilde{N}}
\end{align*}
}with $C_{i}$ of the form stated in the lemma. This gives (\ref{second estimate}).
\end{proof}

\begin{rem}
The explicit form of the constants in Lemma \ref{estimating tail events}
is used in the next section when proving a large deviation principle
for capacities.
\end{rem}

Now we are in a position to complete the proof of Theorem \ref{quasi-sure convergence}.

\begin{proof}[Proof of Theorem \ref{quasi-sure convergence}]

By rewriting (\ref{control of p-var metric as equality}) as
\begin{align}
 & \ d_{p}(\boldsymbol{w},\widetilde{\boldsymbol{w}})\nonumber \\
\leqslant & \ C_{d,p,\gamma}\max\left\{ \rho_{i}(\boldsymbol{w},\widetilde{\boldsymbol{w}})\left(\rho_{j}(\boldsymbol{w})+\rho_{j}(\widetilde{\boldsymbol{w}})\right)^{k}:\left(i,j,k\right)\in\mathbb{N}\times\mathbb{N}\times\mathbb{Z}_{+},i+jk\leqslant3\right\} \label{rewritting control of d_p}
\end{align}
 we only need to show that there exists a positive constant $\beta,$
such that for any $\ \left(i,j,k\right)\in\mathbb{N}\times\mathbb{N}\times\mathbb{Z}_{+}$
satisfying $i+jk\leqslant3$, we have
\begin{align}
 & \ \sum_{m=1}^{\infty}\text{Cap}_{q,N}\left\{ w:\ \rho_{i}\left(\boldsymbol{w}^{(m+1)},\boldsymbol{w}^{(m)}\right)\left(\rho_{j}\left(\boldsymbol{w}^{(m)}\right)+\rho_{j}\left(\boldsymbol{w}^{(m+1)}\right)\right)^{k}>\frac{1}{2^{m\beta}}\right\} \nonumber \\
< & \ \infty.\label{only need to show}
\end{align}
Indeed, if the above result holds, then by Lemma \ref{control of p-var metric as lemma},
we have 
\[
\sum_{m=1}^{\infty}\text{Cap}_{q,N}\left\{ w:\ d_{p}\left(\boldsymbol{w}^{(m)},\boldsymbol{w}^{(m+1)}\right)>C'_{d,p,\gamma}\frac{1}{2^{m\beta}}\right\} <\infty,
\]
where $C'_{d,p,\gamma}$ is some constant depending only on $d,p,\gamma$.
It follows from the quasi-sure version of Borel-Catelli's lemma (the
second part of Proposition \ref{Tchebycheff and Borel-Cantelli})
that 
\[
\text{Cap}_{q,N}\left(\limsup_{m\rightarrow\infty}\left\{ w:\ d_{p}\left(\boldsymbol{w}^{(m)},\boldsymbol{w}^{(m+1)}\right)>C_{d,p,\gamma}'\frac{1}{2^{m\beta}}\right\} \right)=0.
\]
Since
\begin{align*}
\mathcal{A}_{p}^{c}= & \left\{ w:\ \boldsymbol{w}^{(m)}\text{\ is not a Cauchy sequence in }\mbox{under }d_{p}\right\} \\
\subset & \left\{ w:\ \sum_{m=1}^{\infty}d_{p}\left(\boldsymbol{w}^{(m)},\boldsymbol{w}^{(m+1)}\right)=\infty\right\} \\
\subset & \limsup_{m\rightarrow\infty}\left\{ w:\ d_{p}\left(\boldsymbol{w}^{(m)},\boldsymbol{w}^{(m+1)}\right)>C'_{d,p,\gamma}\frac{1}{2^{m\beta}}\right\} ,
\end{align*}
it follows that $\mbox{Cap}_{q,N}(\mathcal{A}_{p}^{c})=0$ which completes
the proof.

Now we prove (\ref{only need to show}).

First consider the case $k>0.$ For any $\alpha,\beta>0,$ we have
\begin{align*}
\text{} & \text{Cap}_{q,N}\left\{ w:\ \rho_{i}\left(\boldsymbol{w}^{(m+1)},\boldsymbol{w}^{(m)}\right)\left(\rho_{j}\left(\boldsymbol{w}^{(m)}\right)+\rho_{j}\left(\boldsymbol{w}^{(m+1)}\right)\right)^{k}>\frac{1}{2^{m\beta}}\right\} \\
\leqslant\text{} & \text{Cap}_{q,N}\left\{ w:\ \rho_{i}\left(\boldsymbol{w}^{(m+1)},\boldsymbol{w}^{(m)}\right)>\frac{1}{2^{m\left(\beta+\alpha\right)}}\right\} \\
 & +\text{Cap}_{q,N}\left\{ w:\ \left(\rho_{j}\left(\boldsymbol{w}^{(m)}\right)+\rho_{j}\left(\boldsymbol{w}^{(m+1)}\right)\right)^{k}>2^{m\alpha}\right\} \\
\leqslant\text{} & \text{Cap}_{q,N}\left\{ w:\ \rho_{i}\left(\boldsymbol{w}^{(m+1)},\boldsymbol{w}^{(m)}\right)>\frac{1}{2^{m\left(\beta+\alpha\right)}}\right\} \\
 & +\text{Cap}_{q,N}\left\{ w:\ \rho_{j}\left(\boldsymbol{w}^{(m)}\right)>2^{\frac{m\alpha}{k}-1}\right\} \\
 & +\text{Cap}_{q,N}\left\{ w:\ \rho_{j}\left(\boldsymbol{w}^{(m+1)}\right)>2^{\frac{m\alpha}{k}-1}\right\} 
\end{align*}
By Lemma \ref{estimating tail events}, for $\theta\in\left(\left(\frac{p(2h+1)}{6}-1\right)^{+},hp-1\right)$,
$\widetilde{N}>\frac{N}{2}\vee\left(2\left(h-\frac{\theta+1}{p}\right)\right)^{-1}$
and $i=1,2,3,$ we have
\begin{align*}
 & \ \text{Cap}_{q,N}\left\{ w:\ \rho_{i}\left(\boldsymbol{w}^{(m+1)},\boldsymbol{w}^{(m)}\right)>\frac{1}{2^{m\left(\beta+\alpha\right)}}\right\} \\
\leqslant & \ C_{i}\left(\frac{1}{2^{m}}\right)^{2i\widetilde{N}\left(h-\frac{\theta+1}{p}\right)-1-2\left(\beta+\alpha\right)\widetilde{N}}.
\end{align*}
Let $\alpha,\beta>0$ be such that 
\begin{equation}
\frac{2\widetilde{N}\left(h-\frac{\theta+1}{p}\right)-1}{2\tilde{N}}>\beta+\alpha>0.\label{choice of alpha and beta}
\end{equation}
It follows easily that 
\begin{equation}
\sum_{m=1}^{\infty}\text{Cap}_{q,N}\left\{ w:\ \rho_{i}\left(\boldsymbol{w}^{(m+1)},\boldsymbol{w}^{(m)}\right)>\frac{1}{2^{m\left(\beta+\alpha\right)}}\right\} <\infty.\label{first term summable}
\end{equation}
Similarly,
\[
\text{Cap}_{q,N}\left\{ w:\ \rho_{j}\left(\boldsymbol{w}^{(m)}\right)>2^{\frac{m\alpha}{k}-1}\right\} \leqslant C_{j}2^{-\frac{m\alpha}{k}-1},
\]
and hence 
\[
\sum_{m=1}^{\infty}\text{Cap}_{q,N}\left\{ w:\ \rho_{j}\left(\boldsymbol{w}^{(m)}\right)>2^{\frac{m\alpha}{k}-1}\right\} <\infty.
\]
Combining with (\ref{first term summable}), we arrive at
\begin{align*}
 & \sum_{m=1}^{\infty}\text{Cap}_{q,N}\left\{ w:\ \rho_{i}\left(\boldsymbol{w}^{(m+1)},\boldsymbol{w}^{(m)}\right)\left(\rho_{j}\left(\boldsymbol{w}^{(m)}\right)+\rho_{j}\left(\boldsymbol{w}^{(m+1)}\right)\right)^{k}>\frac{1}{2^{m\beta}}\right\} \\
< & \ \infty.
\end{align*}

The case of $k=0$ follows directly from (\ref{first term summable}),
since for all $\alpha>0$, 
\[
\left\{ w:\ \rho_{i}\left(\boldsymbol{w}^{(m+1)},\boldsymbol{w}^{(m)}\right)>\frac{1}{2^{m\beta}}\right\} \subset\left\{ w:\ \rho_{i}\left(\boldsymbol{w}^{(m+1)},\boldsymbol{w}^{(m)}\right)>\frac{1}{2^{m\left(\beta+\alpha\right)}}\right\} .
\]

Now the proof is complete.

\end{proof}

\section{Large Deviations for Capacities}

In this section, we apply the previous technique to prove a large
deviation principle for capacities for Gaussian rough paths with long-time
memory. 

Before stating our main result, we first recall the definition of
general LDPs for induced capacities in Polish spaces (see \cite{GR},
\cite{Yoshida}).

Let $(B,H,\mu)$ be an abstract Wiener space. 
\begin{defn}
Let $q>1,N\in\mathbb{N},$ and let $\{T^{\varepsilon}\}$ be a family
of $\mbox{Cap}_{q,N}$-quasi surely defined maps from $B$ to some
Polish space $(X,d).$ We say that the family $\{T^{\varepsilon}\}$
satisfies the $\mbox{Cap}_{q,N}$-\textit{large deviation principle}
(or simply $\mbox{Cap}_{q,N}$-$LDP$) with good rate function $I:\ X\rightarrow[0,\infty]$
if 

(1) $I$ is a good rate function on $X,$ i.e. $I$ is lower semi-continuous
and for every $\alpha>0,$ the level set $\Psi_{I}(\alpha)=\{y\in X:\ I(y)\leqslant\alpha\}$
is compact in $X$;

(2) for every closed subset $C\subset X,$ we have 
\begin{equation}
\limsup_{\varepsilon\rightarrow0}\varepsilon^{2}\log\mbox{Cap}_{q,N}\left\{ w\in B:\ T^{\varepsilon}(w)\in C\right\} \leqslant-\frac{1}{q}\inf_{x\in C}I(x),\label{upper bound for LDP}
\end{equation}
and for ever open subset $G\subset X,$ we have 
\begin{equation}
\liminf_{\varepsilon\rightarrow0}\varepsilon^{2}\log\mbox{Cap}_{q,N}\left\{ w\in B:\ T^{\varepsilon}(w)\in G\right\} \geqslant-\frac{1}{q}\inf_{x\in G}I(x).\label{lower bound for LDP}
\end{equation}

\end{defn}

\begin{rem}
The appearance of the factor $1/q$ comes from the definition of $\mbox{Cap}_{q,N}$,
so 
\begin{equation}
\mbox{Cap}_{q,N}(A)\geqslant\mbox{Cap}_{q,0}(A)=\mathbb{P}(A)^{\frac{1}{q}},\ \forall A\in\mathcal{B}(B).\label{relation between measure and capacity}
\end{equation}
It is consistent with the classical large deviation principle for
probability measures.
\end{rem}

Due to the properties of $(q,N)$-capacity, many important results
for LDPs can be carried through in the capacity setting without much
difficulty, and the proofs are similar to the case of probability
measures. Here we present two fundamental results on transformations
of LDPs for capacities that are crucial for us, which did not appear
in \cite{GR},\cite{Yoshida} and related literatures. 

The first result is the contraction principle.
\begin{thm}
\label{contraction principle}Let $\{T^{\varepsilon}\}$ be a family
of $\mathrm{Cap}_{q,N}$-quasi surely defined maps from $B$ to $(X,d)$
satisfying the $\mathrm{Cap}_{q,N}$-LDP with good rate function $I.$
Let $F$ be a continuous map from $X$ to another Polish space $(Y,d').$
Then the family $\{F\circ T^{\varepsilon}\}$ of $C\mathrm{ap}_{q,N}$-quasi
surely defined maps satisfies the $\mathrm{Cap}_{q,N}$-LDP with good
rate function
\begin{equation}
J(y)=\inf_{x:\ F(x)=y}I(x),\label{rate function in contraction principle}
\end{equation}
where we define $\inf\emptyset=\infty.$ \end{thm}
\begin{proof}
Since $I$ is a good rate function, it is not hard to see that $J$
is lower semi-continuous and also by the continuity of $F$, if $J(y)<\infty$
then the infimum in (\ref{rate function in contraction principle})
is attained at some point $x\in F^{-1}(y).$ Therefore, for any $\alpha>0,$
we have 
\[
\{y\in Y:\ J(y)\leqslant\alpha\}=F\left(\{x\in X:\ I(x)\leqslant\alpha\}\right),
\]
and hence $J$ is a good rate function. The $\mathrm{Cap}_{q,N}$-LDP
(the upper bound (\ref{upper bound for LDP}) and lower bound (\ref{lower bound for LDP}))
for the family $\{F\circ T^{\varepsilon}\}$ under the good rate function
$J$ follows easily from the continuity of $F.$
\end{proof}

The second result is about exponential good approximations. 

We first need the following definition.
\begin{defn}
Let $\{T^{\varepsilon,m}\}$ and $\{T^{\varepsilon}\}$ be two families
of $\mbox{Cap}_{q,N}$-quasi-surely defined maps from $B$ to $(X,d).$
We say that $\{T^{\varepsilon,m}\}$ are \textit{exponentially good
approximations} of $\{T^{\varepsilon}\}$ under $\mbox{Cap}_{q,N}$,
if for any $\lambda>0,$
\begin{equation}
\lim_{m\rightarrow\infty}\limsup_{\varepsilon\rightarrow0}\varepsilon^{2}\log\mbox{Cap}_{q,N}\{w:\ d(T^{\varepsilon,m}(w),T^{\varepsilon}(w))>\lambda\}=-\infty.\label{exponential good approximation}
\end{equation}

\end{defn}

Now we have the following result.
\begin{thm}
\label{exponential approximation}Suppose that for each $m\geqslant1,$
the family $\{T^{\varepsilon,m}\}$ of $\mathrm{Cap}_{q,N}$-quasi-surely
defined maps satisfies the $\mathrm{Cap}_{q,N}$-LDP with good rate
function $I_{m}$ and $\{T^{\varepsilon,m}\}$ are exponentially good
approximations of some family $\{T^{\varepsilon}\}$ of $\mathrm{Cap}_{q,N}$-quasi-surely
defined maps. Suppose further that the function $I$ defined by 
\begin{equation}
I(x)=\sup_{\lambda>0}\liminf_{m\rightarrow\infty}\inf_{y\in B_{x,\lambda}}I_{m}(y),\label{rate function given in exponential approximation}
\end{equation}
where $B_{x,\lambda}$ denotes the open ball $\{y\in X:\ d(y,x)<\lambda\},$
is a good rate function and for every closed set $C\subset X,$ 
\begin{equation}
\inf_{x\in C}I(x)\leqslant\limsup_{m\rightarrow\infty}\inf_{x\in C}I_{m}(x).\label{upper bound criterion in exponential approximations}
\end{equation}
Then $\{T^{\varepsilon}\}$ satisfies the $\mathrm{Cap}_{q,N}$-LDP
with good rate function $I.$\end{thm}
\begin{proof}
\textit{Upper bound.} Let $C$ be a closed subset of $X$. For any
$\lambda>0,$ let $C_{\lambda}=\{x:\ d(x,C)\leqslant\lambda\}.$ Since
\begin{align*}
 & \left\{ w:\ T^{\varepsilon}(w)\in C\right\} \\
\subset & \left\{ w:\ T^{\varepsilon,m}(w)\in C_{\lambda}\right\} \bigcup\left\{ w:\ d\left(T^{\varepsilon,m}(w),T^{\varepsilon}(w)\right)>\lambda\right\} ,
\end{align*}
it follows from the $\mathrm{Cap}_{q,N}$-LDP for $\{T^{\varepsilon,m}\}$
(the upper bound) that 
\begin{align*}
 & \ \limsup_{\varepsilon\rightarrow0}\varepsilon^{2}\log\mathrm{Cap}_{q,N}\left\{ w:\ T^{\varepsilon}(w)\in C\right\} \\
\leqslant & \ \limsup_{\varepsilon\rightarrow0}\varepsilon^{2}\log\mathrm{Cap}_{q,N}\left\{ w:\ T^{\varepsilon,m}(w)\in C_{\lambda}\right\} \\
 & \ \vee\limsup_{\varepsilon\rightarrow0}\varepsilon^{2}\log\mathrm{Cap}_{q,N}\left\{ w:\ d\left(T^{\varepsilon,m}(w),T^{\varepsilon}(w)\right)>\lambda\right\} \\
\leqslant & \ \left(-\frac{1}{q}\inf_{x\in C_{\lambda}}I_{m}(x)\right)\vee\limsup_{\varepsilon\rightarrow0}\varepsilon^{2}\log\mathrm{Cap}_{q,N}\left\{ w:\ d\left(T^{\varepsilon,m}(w),T^{\varepsilon}(w)\right)>\lambda\right\} .
\end{align*}
By letting $m\rightarrow\infty,$ we obtain from (\ref{exponential good approximation})
and (\ref{upper bound criterion in exponential approximations}) that
\begin{align*}
\limsup_{\varepsilon\rightarrow0}\varepsilon^{2}\log\mathrm{Cap}_{q,N}\left\{ w:\ T^{\varepsilon}(w)\in C\right\} \leqslant & -\frac{1}{q}\limsup_{m\rightarrow\infty}\inf_{x\in C_{\lambda}}I_{m}(x)\\
\leqslant & -\frac{1}{q}\inf_{x\in C_{\lambda}}I(x).
\end{align*}
Now the upper bound (\ref{upper bound for LDP}) follows from a basic
property for good rate functions (see \cite{DZ}, Lemma 4.1.6) that
\[
\lim_{\lambda\rightarrow0}\inf_{x\in C_{\lambda}}I(x)=\inf_{x\in C}I(x).
\]

To prove the lower bound (\ref{lower bound for LDP}), we first show
that
\begin{eqnarray}
-\frac{1}{q}I(x) & = & \inf_{\lambda>0}\limsup_{\varepsilon\rightarrow0}\varepsilon^{2}\log\mathrm{Cap}_{q,N}\left\{ w:\ T^{\varepsilon}(w)\in B_{x,\lambda}\right\} \nonumber \\
 & = & \inf_{\lambda>0}\liminf_{\varepsilon\rightarrow0}\varepsilon^{2}\log\mathrm{Cap}_{q,N}\left\{ w:\ T^{\varepsilon}(w)\in B_{x,\lambda}\right\} .\label{Key step in exponential approximation}
\end{eqnarray}
 In fact, since
\begin{align}
 & \left\{ w:\ T^{\varepsilon,m}(w)\in B_{x,\lambda}\right\} \nonumber \\
\subset & \left\{ w:\ T^{\varepsilon}(w)\in B_{x,2\lambda}\right\} \bigcup\left\{ w:\ d(T^{\varepsilon,m}(w),T^{\varepsilon}(w))>\lambda\right\} ,\label{event inclusion in LDP}
\end{align}
we have 
\begin{align*}
 & \ \mathrm{Cap}_{q,N}\left\{ w:\ T^{\varepsilon,m}(w)\in B_{x,\lambda}\right\} \\
\leqslant & \ \mathrm{Cap}_{q,N}\left\{ w:\ T^{\varepsilon}(w)\in B_{x,2\lambda}\right\} +\left\{ w:\ d(T^{\varepsilon,m}(w),T^{\varepsilon}(w))>\lambda\right\} .
\end{align*}
It follows from the $\mathrm{Cap}_{q,N}$-LDP (the lower bound) for
$\{T^{\varepsilon,m}\}$ that 
\begin{align*}
-\frac{1}{q}\inf_{y\in B_{x,\lambda}}I_{m}(y)\leqslant & \ \liminf_{\varepsilon\rightarrow0}\varepsilon^{2}\log\mathrm{Cap}_{q,N}\left\{ w:\ T^{\varepsilon,m}(w)\in B_{x,\lambda}\right\} \\
\leqslant & \ \liminf_{\varepsilon\rightarrow0}\varepsilon^{2}\left(\log\mathrm{Cap}_{q,N}\left\{ w:\ T^{\varepsilon}(w)\in B_{x,2\lambda}\right\} \right.\\
 & \left.\ \vee\log\mathrm{Cap}_{q,N}\left\{ w:\ d(T^{\varepsilon,m}(w),T^{\varepsilon}(w))>\lambda\right\} \right)\\
\leqslant & \ \liminf_{\varepsilon\rightarrow0}\varepsilon^{2}\log\mathrm{Cap}_{q,N}\left\{ w:\ T^{\varepsilon}(w)\in B_{x,2\lambda}\right\} \\
 & \ \vee\limsup_{\varepsilon\rightarrow0}\varepsilon^{2}\log\mathrm{Cap}_{q,N}\left\{ w:\ d(T^{\varepsilon,m}(w),T^{\varepsilon}(w))>\lambda\right\} ,
\end{align*}
and (\ref{exponential good approximation}) implies that
\begin{align*}
-\frac{1}{q}\liminf_{m\rightarrow\infty}\inf_{y\in B_{x,\lambda}}I_{m}(y)\leqslant & \ \liminf_{\varepsilon\rightarrow0}\varepsilon^{2}\log\mathrm{Cap}_{q,N}\left\{ w:\ T^{\varepsilon}(w)\in B_{x,2\lambda}\right\} .
\end{align*}
By taking infimum over $\lambda>0$, we obtain
\[
-\frac{1}{q}I(x)\leqslant\inf_{\lambda>0}\liminf_{\varepsilon\rightarrow0}\varepsilon^{2}\log\mathrm{Cap}_{q,N}\left\{ w:\ T^{\varepsilon}(w)\in B_{x,2\lambda}\right\} .
\]
On the other hand, by exchanging $T^{\varepsilon,m}$ and $T^{\varepsilon}$
in (\ref{event inclusion in LDP}), the same argument yields that
(using the upper bound in the $\mathrm{Cap}_{q,N}$-LDP)
\[
\inf_{\lambda>0}\limsup_{\varepsilon\rightarrow0}\varepsilon^{2}\log\mathrm{Cap}_{q,N}\left\{ w:\ T^{\varepsilon}(w)\in B_{x,\lambda}\right\} \leqslant-\frac{1}{q}I(x).
\]
Therefore, (\ref{Key step in exponential approximation}) follows.

\textit{Lower bound.} Let $G$ be an open subset of $X$. For any
fixed $x\in G,$ take $\lambda>0$ such that $B_{x,\lambda}\subset G.$
It follows from (\ref{Key step in exponential approximation}) that
\begin{align*}
 & \liminf_{\varepsilon\rightarrow0}\varepsilon^{2}\log\mbox{Cap}_{q,N}\left\{ w:\ T^{\varepsilon}(w)\in G\right\} \\
\geqslant & \liminf_{\varepsilon\rightarrow0}\varepsilon^{2}\log\mbox{Cap}_{q,N}\left\{ w:\ T^{\varepsilon}(w)\in B_{x,\lambda}\right\} \\
\geqslant & -\frac{1}{q}I(x).
\end{align*}
Therefore, the lower bound (\ref{lower bound for LDP}) holds.
\end{proof}

Consider the abstract Wiener space $(W,\mathcal{H},\mathbb{P})$ associated
with a Gaussian process satisfying the assumptions in Theorem \ref{quasi-sure convergence}.
According to \cite{FV}, the covariance function of the process has
finite $(1/2h)$-variation in the 2D sense, and $\mathcal{H}$ is
continuously embedded in the space of continuous paths with finite
$(1/2h)$-variation. Therefore, every $h\in\mathcal{H}$ admits a
natural lifting $\boldsymbol{h}$ in $G\Omega_{p}(\mathbb{R}^{d})$
in the sense of iterated Young's integrals (see \cite{Young}). 

Recall that $\mathcal{A}_{p}$ is the set of paths $w\in W$ such
that the lifting $\boldsymbol{w}^{(m)}$ of the dyadic piecewise linear
interpolation of $w$ is a Cauchy sequence under $d_{p}$, and the
map 
\[
F:\ w\in\mathcal{A}_{p}\mapsto\boldsymbol{w}=\left(1,w^{1},\cdots,w^{[p]}\right):=\lim_{m\rightarrow\infty}\boldsymbol{w}^{(m)}\in G\Omega_{p}\left(\mathbb{R}^{d}\right)
\]
is well-defined. For $\varepsilon>0,$ let $T^{\varepsilon}:\ \mathcal{A}_{p}\rightarrow G\Omega_{p}\left(\mathbb{R}^{d}\right)$
be the map defined by 
\[
T^{\varepsilon}(w)=\delta_{\varepsilon}\boldsymbol{w}:=(1,\varepsilon w^{1},\cdots,\varepsilon^{[p]}w^{[p]}).
\]
By Theorem \ref{quasi-sure convergence}, $\mathcal{A}_{p}^{c}$ is
a slim set. Therefore, $T^{\varepsilon}$ is quasi-surely well-defined.

Let 
\begin{equation}
\Lambda(w)=\begin{cases}
\frac{1}{2}\|w\|_{\mathcal{H}}^{2}, & w\in\mathcal{H};\\
\infty, & \mbox{otherwise},
\end{cases}\label{original rate function}
\end{equation}
and define $I:\ G\Omega_{p}\left(\mathbb{R}^{d}\right)\rightarrow[0,\infty]$
by 
\begin{equation}
I(\boldsymbol{w})=\inf\{\Lambda(w):\ w\in\mathcal{A}_{p},\ F(w)=\boldsymbol{w}\}.\label{good rate function}
\end{equation}
We will see later in Lemma \ref{Cameron-Martin convergence} that
$\mathcal{H}\subset\mathcal{A}_{p}$ and hence
\[
I(\boldsymbol{w})=\begin{cases}
\frac{1}{2}\left\Vert \pi_{1}(\boldsymbol{w})_{0,\cdot}\right\Vert _{\mathcal{H}}^{2}, & \mbox{if}\ \pi_{1}(\boldsymbol{w})_{0,\cdot}\in\mathcal{H}\ \mbox{and }\boldsymbol{w}=F\left(\pi_{1}(\boldsymbol{w})_{0,\cdot}\right);\\
\infty, & \mbox{otherwise,}
\end{cases}
\]
where $\pi_{1}$ is the projection onto the first level path.

Now we can state our main result of this section.
\begin{thm}
\label{LDP for capacity}For any $q>1,$ $N\in\mathbb{N},$ the family
$\{T^{\varepsilon}\}$ of $\mathrm{Cap}_{q,N}$-quasi-surely defined
maps from $W$ to $G\Omega_{p}\left(\mathbb{R}^{d}\right)$ satisfies
the $\mathrm{Cap}_{q,N}$-LDP with good rate function $I$.
\end{thm}

In particular, since the projection map from $G\Omega_{p}\left(\mathbb{R}^{d}\right)$
onto the first level path is continuous, we immediately obtain the
following result of Yoshida \cite{Yoshida} in the case of Gaussian
processes with long-time memory. 
\begin{cor}
The family of maps $\{\varepsilon w\}$ satisfies the $\mathrm{Cap}_{q,N}$-LDP
with good rate function $\Lambda.$
\end{cor}

Moreover, according to the universal limit theorem (Theorem \ref{Universal Limit Theorem})
and the contraction principle (Theorem \ref{contraction principle}),
a direct corollary of Theorem \ref{LDP for capacity} is the LDPs
for capacities for solutions to differential equations driven by Gaussian
rough paths with long-time memory. This generalizes the classical
Freidlin-Wentzell theory for diffusion measures in the quasi-sure
and rough path setting. Here we are again taking the advantage of
working in the stronger topology (the $p$-variation topology), under
which we have nice stability for differential equations.

It should be pointed out that the lifting map $F$, which can be regarded
as the pathwise solution to a differential equation driven by $w$
with a polynomial one form, is \textit{not} continuous under the uniform
topology (see \cite{LCL}, \cite{LQ}). Therefore the contraction
principle cannot be applied directly in our context. A standard way
of getting around this difficulty, as in \cite{LQZ} for Brownian
motion and \cite{MS} for fractional Brownian motion in the case of
LDPs for probability measures, is to construct exponentially good
approximations by using dyadic piecewise linear interpolation. Here
we adopt the same idea in the capacity setting.

Let $T^{\varepsilon,m}:\ W\rightarrow G\Omega_{p}(\mathbb{R}^{d})$
be the map given by $T^{\varepsilon,m}(w)=\delta_{\varepsilon}\boldsymbol{w}^{(m)}.$
The proof of Theorem \ref{exponential approximation} essentially
consists of two parts: show that the family $\left\{ T^{\varepsilon,m}\right\} $
satisfies a $\mathrm{Cap}_{q,N}$-LDP and show that $\left\{ T^{\varepsilon,m}\right\} $
are exponentially good approximations of $\{T^{\varepsilon}\}$ under
$\mathrm{Cap}_{q,N}$. 

We first need to establish the $\mathrm{Cap}_{q,N}$-LDP for $\left\{ T^{\varepsilon,m}\right\} $,
and we begin with considering the standard finite dimensional abstract
Wiener space.

Let $\mu$ be the standard Gaussian measure on $\mathbb{R}^{n}.$
In this case, the Cameron-Martin space is just $\mathbb{R}^{n}$ equipped
with the standard Euclidean inner product. For clarity we use the
notation $\mathrm{Cap}_{q,N}^{(n)}$ to emphasize that the capacity
is defined on $\mathbb{R}^{n}$. Now we have the following result.
\begin{prop}
The family $\{\varepsilon x\}$ satisfies the $\mathrm{Cap}_{q,N}^{(n)}$-LDP
with good rate function 
\[
J(x)=\frac{|x|^{2}}{2},\ x\in\mathbb{R}^{n}.
\]
\end{prop}
\begin{proof}
The lower bound follows immediately from the simple relation in (\ref{relation between measure and capacity})
and the classical LDP for the family $\{\mu\left(\varepsilon^{-1}dx\right)\}$
of probability measures. It suffices to establish the upper bound.

We first prove the following inequality for the one dimensional case:
\begin{equation}
\limsup_{\varepsilon\rightarrow0}\varepsilon^{2}\log\mathrm{Cap}_{q,N}^{(1)}\{x:\ \varepsilon x>b\}\leqslant-\frac{1}{2q}b^{2},\label{one dimensional case for half interval}
\end{equation}
where $b>0.$ In fact, for any $\lambda>0,$ define the non-negative
function 
\[
f(x)=\mathrm{e}^{\lambda\varepsilon x-\lambda b},\ x\in\mathbb{R}^{1}.
\]
Obviously $f\in\mathbb{D}_{N}^{q}$, and $f\geqslant1$ on $\{x:\ \varepsilon x>b\}.$
Therefore, by the definition of capacity we have
\begin{eqnarray*}
\mathrm{Cap}_{q,N}^{(1)}\{x:\ \varepsilon x>b\} & \leqslant & \|f\|_{q,N}\\
 & \leqslant & \sum_{i=0}^{N}\left(\int_{\mathbb{R}^{1}}\left|f^{(i)}\right|^{q}\mu(dx)\right)^{\frac{1}{q}}\\
 & = & \sum_{i=0}^{N}\left(\int_{\mathbb{R}^{1}}(\lambda\varepsilon)^{qi}\mathrm{e}^{q\lambda\varepsilon x-q\lambda}\frac{1}{\sqrt{2\pi}}\mathrm{e}^{-\frac{x^{2}}{2}}dx\right)^{\frac{1}{q}}\\
 & = & \sum_{i=0}^{N}(\lambda\varepsilon)^{i}\mathrm{e}^{\frac{q}{2}(\lambda\varepsilon)^{2}-\lambda b}.
\end{eqnarray*}
It follows that 
\[
\varepsilon^{2}\log\mathrm{Cap}_{q,N}^{(1)}\{x:\ \varepsilon x>b\}\leqslant\varepsilon^{2}\log N+\max_{0\leqslant i\leqslant N}\left\{ i\varepsilon^{2}\log(\lambda\varepsilon)\right\} +\frac{q}{2}(\lambda\varepsilon^{2})^{2}-\lambda\varepsilon^{2}b.
\]
Now take $\lambda=b/(q\varepsilon^{2})$, then we have 
\[
\varepsilon^{2}\log\mathrm{Cap}_{q,N}^{(1)}\{x:\ \varepsilon x>b\}\leqslant\varepsilon^{2}\log N+\max_{0\leqslant i\leqslant N}\left\{ i\varepsilon^{2}\log\left(\frac{b}{q\varepsilon}\right)\right\} -\frac{b^{2}}{2q},
\]
and therefore (\ref{one dimensional case for half interval}) holds.
Apparently (\ref{one dimensional case for half interval}) still holds
if $\{x:\ \varepsilon x>b\}$ is replaced by $\{x:\ \varepsilon x\geqslant b\},$
and a similar inequality holds for $\{x:\ \varepsilon x\leqslant a\}$
for $a<0.$

Now we come back to the $n$-dimensional case. 

Firstly, consider an open ball $B(a,r)\subset\mathbb{R}^{n}.$ For
any $\lambda\in\mathbb{R}^{n},$ consider the non-negative function
\[
f(x)=\mathrm{e}^{\langle\lambda,\varepsilon x\rangle+|\lambda|r-\langle\lambda,a\rangle},\ x\in\mathbb{R}^{n}.
\]
Then apparently we have $f\in\mathbb{D}_{N}^{q}$. Moreover, from
the fact that 
\[
\langle\lambda,a\rangle-|\lambda|r=\inf_{y\in B(a,r)}\langle\lambda,y\rangle,
\]
we have $f\geqslant1$ on $\{x:\ \varepsilon x\in B(a,r)\}.$ Therefore,
similarly as before we have 
\begin{eqnarray*}
\mathrm{Cap}_{q,N}^{(n)}\{x:\ \varepsilon x\in B(a,r)\} & \leqslant & \|f\|_{q,N}\\
 & \leqslant & \sum_{i=0}^{N}\left(\int_{\mathbb{R}^{n}}\left|D^{i}f\right|^{q}\mu(dx)\right)^{\frac{1}{q}}\\
 & \leqslant & \sum_{i=0}^{N}(n|\lambda|\varepsilon)^{i}\mathrm{e}^{\frac{1}{2}q(|\lambda|\varepsilon)^{2}+|\lambda|r-\langle\lambda,a\rangle}
\end{eqnarray*}
and 
\begin{eqnarray*}
 &  & \varepsilon^{2}\log\mathrm{Cap}_{q,N}^{(n)}\{x:\ \varepsilon x\in B(a,r)\}\\
 & \leqslant & \varepsilon^{2}\log N+\max_{0\leqslant i\leqslant N}\left\{ i\varepsilon^{2}\log(n|\lambda\varepsilon|)\right\} +\frac{q}{2}\left(|\lambda|\varepsilon^{2}\right)^{2}\\
 &  & +|\lambda|\varepsilon^{2}r-\langle\varepsilon^{2}\lambda,a\rangle.
\end{eqnarray*}
Note that the function $\frac{q}{2}\left(|\lambda|\varepsilon^{2}\right)^{2}+|\lambda|\varepsilon^{2}r-\langle\varepsilon^{2}\lambda,a\rangle$
attains its minimum at 
\[
\lambda=\frac{(|a|-r)^{+}}{q\varepsilon^{2}|a|}a,
\]
By taking this $\lambda$ and letting $\varepsilon\rightarrow0$,
we arrive at 
\[
\limsup_{\varepsilon\rightarrow0}\varepsilon^{2}\log\mathrm{Cap}_{q,N}^{(n)}\{x:\ \varepsilon x\in B(a,r)\}\leqslant-\frac{1}{2q}\left((|a|-r)^{+}\right)^{2}=-\frac{1}{q}\inf_{y\in B(a,r)}J(y).
\]

Secondly, let $K$ be a compact subset of $\mathbb{R}^{n}.$ Then
for any $\delta>0,$ we can find a finite cover of $K$ by open balls
$\{B(a_{i},\delta)\}_{1\leqslant i\leqslant k(\delta)}$ where each
$a_{i}\in K.$ It follows that
\begin{align*}
 & \limsup_{\varepsilon\rightarrow0}\varepsilon^{2}\log\mathrm{Cap}_{q,N}\{x:\ \varepsilon x\in K\}\\
\leqslant & \limsup_{\varepsilon\rightarrow0}\varepsilon^{2}\left(\log k(\delta)+\max_{1\leqslant i\leqslant k(\delta)}\log\mathrm{Cap}_{q,N}^{(n)}\{x:\ \varepsilon x\in B(a_{i},\delta)\}\right)\\
\leqslant & \max_{1\leqslant i\leqslant k(\delta)}\left(-\frac{1}{q}\inf_{y\in B(a_{i},\delta)}J(y)\right)\\
\leqslant & -\frac{1}{q}\inf_{y\in B(K,\delta)}J(y),
\end{align*}
where $B(K,\delta):=\{x:\ \mathrm{dist}(x,K)<\delta\}.$ By letting
$\delta\rightarrow0$ we obtain the upper bound result for the compact
set $K.$ 

Finally, let $C$ be an arbitrary closed subset of $\mathbb{R}^{n}.$
For $\rho>0,$ let 
\[
H_{\rho}=\{x:\ \left|x^{i}\right|\leqslant\rho\ \mathrm{for\ all}\ i\}.
\]
Then we have 
\[
\mathrm{Cap}_{q,N}^{(n)}\{x:\ \varepsilon x\in C\}\leqslant\mathrm{Cap}_{q,N}^{(n)}\{x:\ \varepsilon x\in C\bigcap H_{\rho}\}+\sum_{i=1}^{n}\mathrm{Cap}_{q,N}^{(n)}\{x:\ \varepsilon\left|x^{i}\right|>\rho\}.
\]
On the other hand, from the definition of capacity, we have (see also
the proof of the following Corollary \ref{Cap-LDP for T^epsilon,m}):
\[
\mathrm{Cap}_{q,N}^{(n)}\{x:\ \varepsilon\left|x^{i}\right|>\rho\}\leqslant\mathrm{Cap}_{q,N}^{(1)}\{x\in\mathbb{R}^{1}:\ \varepsilon|x|>\rho\}.
\]
Combining with the upper bound result for compact sets and (\ref{one dimensional case for half interval}),
we arrive at
\[
\limsup_{\varepsilon\rightarrow0}\varepsilon^{2}\log\mathrm{Cap}_{q,N}^{(n)}\{x:\ \varepsilon x\in C\}\leqslant\max\left\{ -\frac{1}{q}\inf_{y\in C\bigcap H_{\rho}}J(y),-\frac{1}{q}\rho^{2}\right\} ,\ \forall\rho>0.
\]
The upper bound result for $C$ follows from letting $\rho\rightarrow\infty.$
\end{proof}

Now consider the situation where $\nu$ is a general non-degenerate
Gaussian measure on $\mathbb{R}^{n}$ with covariance matrix $\Sigma.$
In this case the Cameron-Martin space $\mathcal{H}=\mathbb{R}^{n}$
but with inner product 
\[
\langle h_{1},h_{2}\rangle=h_{1}^{T}\Sigma^{-1}h_{2}.
\]
Moreover, the Cameron-Martin embedding $\iota:\ \mathcal{H}\rightarrow\mathbb{R}^{n}$
is just the identity map but the dual embedding $\iota^{*}:\ \mathbb{R}^{n}\rightarrow\mathcal{H}^{*}\cong\mathcal{H}$
is given by 
\[
\iota^{*}(\lambda)=\Sigma\lambda,\ \lambda\in\mathbb{R}^{n}.
\]
Therefore, if we write $\Sigma=QQ^{T}$ for some non-degenerate matrix
$Q$, it follows from the definition of Sobolev spaces and change
of variables that 
\[
\mathrm{Cap}_{q,N}^{\nu}(A)=\mathrm{Cap}_{q,N}^{\mu}\left(Q^{-1}A\right),\ \forall A\subset\mathbb{R}^{n},
\]
where the L.H.S. is the capacity for $\nu$ and the R.H.S. is the
capacity for the standard Gaussian measure $\mu.$ In other words,
capacities for non-degenerate Gaussian measures on $\mathbb{R}^{n}$
are all equivalent. As a consequence, we conclude that the family
$\{\varepsilon x\}$ satisfies the $\mathrm{Cap}_{q,N}^{\nu}$-LDP
with good rate function 
\[
J(y)=\frac{1}{2}\|y\|_{\mathcal{H}}^{2}=\frac{1}{2}y^{T}\Sigma^{-1}y,\ y\in\mathbb{R}^{n}.
\]
The case of degenerate Gaussian measures follows easily by restriction
on the maximal invariant subspace on which the covariance matrix is
positive definite.

A direct consequence of the previous discussion is the following.
\begin{cor}
\label{Cap-LDP for T^epsilon,m}For each $m\geqslant1,$ the family
$\{T^{\varepsilon,m}\}$ satisfies the $\mathrm{Cap}_{q,N}$-LDP with
good rate function
\begin{equation}
I_{m}(\boldsymbol{w})=\inf\left\{ J_{m}(x):\ x\in\left(\mathbb{R}^{d}\right)^{2^{m}}:\ \Phi_{m}(x)=\boldsymbol{w}\right\} ,\ \boldsymbol{w}\in G\Omega_{p}\left(\mathbb{R}^{d}\right),\label{Rate function in the discrete case}
\end{equation}
where $J_{m}(x)$ is the good rate function for the Gaussian measure
$\nu_{m}$ on $\left(\mathbb{R}^{d}\right)^{2^{m}}$ induced by $\left(w_{t_{m}^{1}},\cdots,w_{t_{m}^{2^{m}}}\right),$
and $\Phi_{m}$ is the map sending each $x\in\left(\mathbb{R}^{d}\right)^{2^{m}}$
to the lifting of the dyadic piecewise linear interpolation associated
with $x.$\end{cor}
\begin{proof}
Since $\Phi_{m}$ is continuous under the Euclidean and $p$-variation
topology respectively, the result follows immediately from the contraction
principle (Theorem \ref{contraction principle}) once we have established
the $\mathrm{Cap}_{q,N}$-LDP for the family $\varepsilon\pi_{m}:\ W\rightarrow\left(\mathbb{R}^{d}\right)^{2^{m}}$
where $\pi^{m}$ is defined by 
\[
\pi_{m}(w)=\left(w_{t_{m}^{1}},\cdots,w_{t_{m}^{2^{m}}}\right),\ w\in W,
\]
with good rate function $J_{m}.$

To see this, first notice again that the lower bound follows from
the relation (\ref{relation between measure and capacity}) and the
classical LDP for finite dimensional Gaussian measures. Moreover,
let $U$ be an open subset of $\left(\mathbb{R}^{d}\right)^{2^{m}}$
and let $f\in\mathbb{D}_{N}^{q}\left(\nu_{m}\right)$ be a function
such that for $\nu_{m}$-almost surely 
\[
f\geqslant1\ \mathrm{on}\ U,\ f\geqslant0\ \mathrm{on}\ \left(\mathbb{R}^{d}\right)^{2^{m}},
\]
$\mathbb{D}_{N}^{q}\left(\nu_{m}\right)$ is the Sobolev space over
$\left(\mathbb{R}^{d}\right)^{2^{m}}$ associated with $\nu_{m}$.
Define 
\[
g(w)=f\left(w_{t_{m}^{1}},\cdots,w_{t_{m}^{2^{m}}}\right),\ w\in W.
\]
Apparently $g\in\mathbb{D}_{N}^{q}$, and for $\mathbb{P}$-almost
surely
\[
g\geqslant1\ \mathrm{on}\ \pi_{m}^{-1}U,\ g\geqslant0\ \mathrm{on}\ W.
\]
Moreover, since $\|g\|_{q,N}=\|f\|_{q,N;\nu_{m}}$, we know that 
\[
\mathrm{Cap}_{q,N}\left(\pi_{m}^{-1}U\right)\leqslant\|f\|_{q,N;\nu_{m}}.
\]
By taking infimum over all such $f,$ we obtain 
\[
\mathrm{Cap}_{q,N}\left(\pi_{m}^{-1}U\right)\leqslant\mathrm{Cap}_{q,N}^{\nu_{m}}(U).
\]
Now the upper bound result follows from the $\mathrm{Cap}_{q,N}$-LDP
for the family $\{\nu_{m,\varepsilon}:=\nu_{m}(\varepsilon^{-1}dx)\}$
of probability measure and a simple limiting argument.
\end{proof}

\begin{rem}
There is an equivalent way of expressing the rate function $I_{m},$
which is very convenient for us to prove our main result of Theorem
\ref{LDP for capacity}. In fact, from classical LDP results for Gaussian
measures (see for example \cite{DS}), we know that the family $\left\{ \mathbb{P}_{\varepsilon}:=\mathbb{P}\left(\varepsilon^{-1}dw\right)\right\} $
of probability measures on $W$ satisfies the LDP with good rate function
$\Lambda$ given by (\ref{original rate function}). Moreover, the
map $\Psi_{m}:\ W\rightarrow G\Omega_{p}(\mathbb{R}^{d})$ defined
by $\Psi_{m}(w)=\boldsymbol{w}^{(m)}$ is continuous under the uniform
and $p$-variation topology respectively. Therefore, according to
the classical contraction principle, the family $\{\mathbb{P}_{\varepsilon}\circ\Psi_{m}^{-1}\}$
of probability measures on $G\Omega_{p}\left(\mathbb{R}^{d}\right)$
satisfies the LDP with good rate function 
\begin{equation}
I_{m}'(\boldsymbol{w})=\inf\left\{ \Lambda(w):\ w\in W,\ \Psi_{m}(w)=\boldsymbol{w}\right\} ,\ \boldsymbol{w}\in G\Omega_{p}\left(\mathbb{R}^{d}\right).\label{approximating rate function}
\end{equation}
On the other hand, the same argument implies that the family $\left\{ \nu_{m,\varepsilon}\circ\Phi_{m}^{-1}\right\} $
of probability measures on $G\Omega_{p}\left(\mathbb{R}^{d}\right)$
satisfies the LDP with good rate function $I_{m}$ given by (\ref{Rate function in the discrete case}).
Observe that $\mathbb{P}_{\varepsilon}\circ\Psi_{m}^{-1}=\nu_{m,\varepsilon}\circ\Phi_{m}^{-1}.$
By the uniqueness of rate functions (see \cite{DZ}, Chapter 4, Lemma
4.1.4), we conclude that $I_{m}=I_{m}'.$
\end{rem}

\begin{rem}
Of course we can apply Yoshida's result directly with the contraction
principle to obtain the $\mathrm{Cap}_{q,N}$-LDP for the family $\left\{ T^{\varepsilon,m}\right\} $
with good rate function $I'_{m}.$ Here we do not proceed in this
way so that in the end our result yields Yoshida's one as a corollary,
and our proof relies only on basic properties of capacities and finite
dimensional Gaussian spaces.
\end{rem}

The second main ingredient of proving Theorem \ref{LDP for capacity}
is the following.
\begin{lem}
\label{Key Lemma for LDP}For any $q>1,N\in\mathbb{N}$ and $\lambda>0$,
we have 
\[
\lim_{m\rightarrow\infty}\limsup_{\varepsilon\rightarrow0}\varepsilon^{2}\log\mathrm{Cap}_{q,N}\left\{ w:\ d_{p}\left(\delta_{\varepsilon}\boldsymbol{w}^{(m)},\delta_{\varepsilon}\boldsymbol{w}\right)>\lambda\right\} =-\infty.
\]
Therefore, $\{T^{\varepsilon,m}\}$ are exponentially good approximations
of $\{T^{\varepsilon}\}$ under $\mathrm{Cap}_{q,N}$.\end{lem}
\begin{proof}
For any $\beta>0,$ since 
\begin{align*}
 & \left\{ w:\ d_{p}\left(\delta_{\varepsilon}\boldsymbol{w}^{(m)},\delta_{\varepsilon}\boldsymbol{w}\right)>\lambda\right\} \\
\subset & \left\{ w:\ \sum_{l=m}^{\infty}d_{p}\left(\delta_{\varepsilon}\boldsymbol{w}^{(l)},\delta_{\varepsilon}\boldsymbol{w}^{(l+1)}\right)>\lambda\right\} \\
\subset & \bigcup_{l=m}^{\infty}\left\{ w:\ d_{p}\left(\delta_{\varepsilon}\boldsymbol{w}^{(l)},\delta_{\varepsilon}\boldsymbol{w}^{(l+1)}\right)>\frac{\lambda}{C_{\beta}}\cdot\frac{1}{2^{(l-m)\beta}}\right\} ,
\end{align*}
we have
\begin{align*}
 & \ \mbox{Cap}_{q,N}\left\{ w:\ d_{p}\left(\delta_{\varepsilon}\boldsymbol{w}^{(m)},\delta_{\varepsilon}\boldsymbol{w}\right)>\lambda\right\} \\
\leqslant & \ \sum_{l=m}^{\infty}\mbox{Cap}_{q,N}\left\{ w:\ d_{p}\left(\delta_{\varepsilon}\boldsymbol{w}^{(l)},\delta_{\varepsilon}\boldsymbol{w}^{(l+1)}\right)>\frac{\lambda}{C_{\beta}}\cdot\frac{1}{2^{(l-m)\beta}}\right\} ,
\end{align*}
where $C_{\beta}:=\sum_{k=0}^{\infty}2^{-\beta k}.$ It then follows
from (\ref{rewritting control of d_p}) that for any $\alpha>0$,
\begin{align}
 & \ \mbox{Cap}_{q,N}\left\{ w:\ d_{p}\left(\delta_{\varepsilon}\boldsymbol{w}^{(m)},\delta_{\varepsilon}\boldsymbol{w}\right)>\lambda\right\} \nonumber \\
\leqslant & \ \sum_{i=1}^{3}\sum_{l=m}^{\infty}\mbox{Cap}_{q,N}\left\{ w:\ \rho_{i}\left(\delta_{\varepsilon}\boldsymbol{w}^{(l)},\delta_{\varepsilon}\boldsymbol{w}^{(l+1)}\right)>\frac{\lambda}{C_{d,p,\gamma,\beta}}\frac{1}{2^{(l-m)\beta}}\right\} \nonumber \\
 & \ +\sum_{\substack{i,j,k\geqslant1\\
i+jk\leqslant3
}
}\sum_{l=m}^{\infty}\left(\mbox{Cap}_{q,N}\left\{ w:\ \rho_{i}\left(\delta_{\varepsilon}\boldsymbol{w}^{(l)},\delta_{\varepsilon}\boldsymbol{w}^{(l+1)}\right)>\frac{\lambda}{C_{d,p,\gamma,\beta}}\frac{2^{m\beta}}{2^{l(\alpha+\beta)}}\right\} \right.\nonumber \\
 & \ +\mbox{Cap}_{q,N}\left\{ w:\ \rho_{j}\left(\delta_{\varepsilon}\boldsymbol{w}^{(l)}\right)>2^{\frac{l\alpha}{k}-1}\right\} \nonumber \\
 & \ +\mbox{Cap}_{q,N}\left\{ w:\ \rho_{j}\left(\delta_{\varepsilon}\boldsymbol{w}^{(l+1)}\right)>2^{\frac{l\alpha}{k}-1}\right\} ,\label{several terms}
\end{align}
where $C_{d,p,\gamma,\beta}$ is a constant depending only on $p,d,\gamma,\beta$. 

Similar to the proof of Theorem \ref{quasi-sure convergence}, we
estimate each term on the R.H.S. of (\ref{several terms}). Here we
choose $\alpha,\beta$ in exactly the same way as in the proof of
Theorem \ref{quasi-sure convergence}, namely, by (\ref{choice of alpha and beta}).
It should be pointed out that the choice of $\alpha,\beta$ can be
made independent of $\widetilde{N},$ since $\theta\in\left(\left(\frac{p(2h+1)}{6}-1\right)^{+},hp-1\right).$

Firstly, it follows from Lemma \ref{estimating tail events} that
for $i=1,2,3,$
\begin{align}
 & \ \mbox{Cap}_{q,N}\left\{ w:\ \rho_{i}\left(\delta_{\varepsilon}\boldsymbol{w}^{(l)},\delta_{\varepsilon}\boldsymbol{w}^{(l+1)}\right)>\frac{\lambda}{C_{d,p,\gamma,\beta}}\frac{1}{2^{(l-m)\beta}}\right\} \nonumber \\
= & \ \mbox{Cap}_{q,N}\left\{ w:\ \rho_{i}\left(\boldsymbol{w}^{(l)},\boldsymbol{w}^{(l+1)}\right)>\frac{\lambda}{C_{d,p,\gamma,\beta}}\frac{\varepsilon^{-i}}{2^{(l-m)\beta}}\right\} \nonumber \\
\leqslant & \ C_{1}C_{2}^{\widetilde{N}}g\left(\widetilde{N};N\right)\widetilde{N}^{i\widetilde{N}}\cdot\left(\frac{\lambda}{C_{d,p,\gamma,\beta}}\frac{\varepsilon^{-i}}{2^{(l-m)\beta}}\right)^{-2\widetilde{N}}\cdot\left(\frac{1}{2^{l}}\right)^{2i\widetilde{N}\left(h-\frac{\theta+1}{p}\right)-1}\nonumber \\
= & \ C_{1}C_{3}^{\widetilde{N}}g\left(\widetilde{N};N\right)\left(\widetilde{N}\varepsilon^{2}\right)^{i\widetilde{N}}\cdot\frac{1}{2^{2m\widetilde{N}\beta}}\left(\frac{1}{2^{l}}\right)^{2i\widetilde{N}\left(h-\frac{\theta+1}{p}\right)-1-2\widetilde{N}\beta},\label{first term}
\end{align}
where $C_{3}=C_{2}\left(\frac{\lambda}{C_{d,p,\gamma,\beta}}\right)^{-2}.$
Note that by the choice of $\beta,$ the R.H.S. of (\ref{first term})
is summable over $l,$ and it follows that 
\begin{align*}
 & \ \sum_{l=m}^{\infty}\mbox{Cap}_{q,N}\left\{ w:\ \rho_{i}\left(\delta_{\varepsilon}\boldsymbol{w}^{(l)},\delta_{\varepsilon}\boldsymbol{w}^{(l+1)}\right)>\frac{\lambda}{C_{d,p,\gamma,\beta}}\frac{1}{2^{(l-m)\beta}}\right\} \\
\leqslant & \ C_{4}C_{3}^{\widetilde{N}}g\left(\widetilde{N};N\right)\left(\widetilde{N}\varepsilon^{2}\right)^{i\widetilde{N}}\cdot\left(\frac{1}{2^{m}}\right)^{2i\widetilde{N}\left(h-\frac{\theta+1}{p}\right)-1},
\end{align*}
where $C_{4}=C_{1}\left(1-2^{-\left(2i\widetilde{N}\left(h-\frac{\theta+1}{p}\right)-1-2\widetilde{N}\beta\right)}\right)^{-1}.$
By taking $\widetilde{N}=\left[\varepsilon^{-2}\right]$ for $\varepsilon$
small enough, it is easy to see that 
\begin{align*}
 & \limsup_{\varepsilon\rightarrow0}\varepsilon^{2}\log\left(\sum_{l=m}^{\infty}\mbox{Cap}_{q,N}\left\{ w:\ \rho_{i}\left(\delta_{\varepsilon}\boldsymbol{w}^{(l)},\delta_{\varepsilon}\boldsymbol{w}^{(l+1)}\right)>\frac{\lambda}{C_{d,p,\gamma,\beta}}\frac{1}{2^{(l-m)\beta}}\right\} \right)\\
= & \log C_{3}+2i\left(h-\frac{\theta+1}{p}\right)\log\left(\frac{1}{2^{m}}\right).
\end{align*}
Therefore, we have
\begin{align*}
 & \lim_{m\rightarrow\infty}\limsup_{\varepsilon\rightarrow0}\varepsilon^{2}\log\left(\sum_{l=m}^{\infty}\mbox{Cap}_{q,N}\left\{ w:\ \rho_{i}\left(\delta_{\varepsilon}\boldsymbol{w}^{(l)},\delta_{\varepsilon}\boldsymbol{w}^{(l+1)}\right)>\frac{\lambda}{C_{d,p,\gamma,\beta}}\frac{1}{2^{(l-m)\beta}}\right\} \right)\\
= & -\infty.
\end{align*}

Again by the choice of $\alpha,\beta$ and by taking $\widetilde{N}=\left[\varepsilon^{-2}\right],$
the same computation based on Lemma \ref{estimating tail events}
yields that 
\begin{align*}
 & \lim_{m\rightarrow\infty}\limsup_{\varepsilon\rightarrow0}\varepsilon^{2}\log\left(\sum_{l=m}^{\infty}\mbox{Cap}_{q,N}\left\{ w:\ \rho_{i}\left(\delta_{\varepsilon}\boldsymbol{w}^{(l)},\delta_{\varepsilon}\boldsymbol{w}^{(l+1)}\right)>\frac{\lambda}{C_{d,p,\gamma,\beta}}\frac{2^{m\beta}}{2^{l(\alpha+\beta)}}\right\} \right)\\
= & \lim_{m\rightarrow\infty}\limsup_{\varepsilon\rightarrow0}\varepsilon^{2}\log\left(\sum_{l=m}^{\infty}\mbox{Cap}_{q,N}\left\{ w:\ \rho_{j}\left(\delta_{\varepsilon}\boldsymbol{w}^{(l)}\right)>2^{\frac{l\alpha}{k}-1}\right\} \right)\\
= & \lim_{m\rightarrow\infty}\limsup_{\varepsilon\rightarrow0}\varepsilon^{2}\log\left(\sum_{l=m}^{\infty}\mbox{Cap}_{q,N}\left\{ w:\ \rho_{j}\left(\delta_{\varepsilon}\boldsymbol{w}^{(l+1)}\right)>2^{\frac{l\alpha}{k}-1}\right\} \right)\\
= & -\infty,
\end{align*}
for $i,j,k\geqslant1$ with $i+jk\leqslant3.$ 

Now the desired result follows easily.
\end{proof}

In order to apply Theorem \ref{exponential approximation}, we need
the following convergence result in \cite{FV} for Cameron-Martin
paths.
\begin{lem}
\label{Cameron-Martin convergence} For any $\alpha>0,$ we have 
\[
\lim_{m\rightarrow\infty}\sup_{\{h\in\mathcal{H}:\ \|h\|_{\mathcal{H}}\leqslant\alpha\}}d_{p}\left(\boldsymbol{h}^{(m)},\boldsymbol{h}\right)=0.
\]
In particular, $\mathcal{H}$ is contained in $\mathcal{A}_{p}.$
\end{lem}
Now we are in a position to prove Theorem \ref{LDP for capacity}.

\begin{proof}[Proof of Theorem \ref{LDP for capacity}]

It suffices to show that the function $I$ given by (\ref{good rate function})
coincides with the one given by (\ref{rate function given in exponential approximation}),
and it satisfies all conditions in Theorem \ref{exponential approximation}.
Here we use $I_{m}'$ given by (\ref{approximating rate function})
for the rate function of $\{T^{\varepsilon,m}\}.$ 

Firstly, by Lemma \ref{Cameron-Martin convergence} it is easy to
see that the lifting map $F$ is continuous on each level set $\{w:\ \Lambda(w)\leqslant\alpha\}\subset\mathcal{H}\subset\mathcal{A}_{p}$
of $\Lambda.$ It follows from the definition of $I$ that 
\[
F(\{w:\ \Lambda(w)\leqslant\alpha\})=\{\boldsymbol{w}:\ I(\boldsymbol{w})\leqslant\alpha\},
\]
which then implies that $I$ is a good rate function. 

Now we show that for any closed subset $C\subset G\Omega_{p}\left(\mathbb{R}^{d}\right),$
we have 
\begin{equation}
\inf_{\boldsymbol{w}\in C}I(\boldsymbol{w})\leqslant\liminf_{m\rightarrow\infty}\inf_{\boldsymbol{w}\in C}I'_{m}(\boldsymbol{w}).\label{verifying conditions}
\end{equation}
In fact, let $\gamma_{m}=\inf_{\boldsymbol{w}\in C}I'_{m}(\boldsymbol{w})=\inf_{w\in\Psi_{m}^{-1}(C)}\Lambda(w).$
We only consider the nontrivial case $\liminf_{m\rightarrow\infty}\gamma_{m}=\alpha<\infty,$
and without loss of generality we  assume that $\lim_{m\rightarrow\infty}\gamma_{m}=\alpha.$
Since $\Lambda$ is a good rate function, we know that the infimum
over the closed subset $\Psi_{m}^{-1}(C)\subset W$ is attainable.
Therefore, there exists $w_{m}\in W$ such that $\Psi_{m}(w_{m})\in C$
and $\gamma_{m}=\Lambda(w_{m}).$ It follows from Lemma \ref{Cameron-Martin convergence}
that for any fixed $\lambda>0,$ $F(w_{m})\in C_{\lambda}$ when $m$
is large, where $C_{\lambda}:=\{\boldsymbol{w}:\ d_{p}(\boldsymbol{w},C)\leqslant\lambda\}$.
Consequently, when $m$ is large, we have 
\[
\inf_{\boldsymbol{w}\in C_{\lambda}}I(\boldsymbol{w})\leqslant I(F(w_{m}))=\Lambda(w_{m})=\gamma_{m},
\]
and hence 
\[
\inf_{\boldsymbol{w}\in C_{\lambda}}I(\boldsymbol{w})\leqslant\alpha.
\]
(\ref{verifying conditions}) then follows easily from \cite{DZ},
Chapter 4, Lemma 4.1.6. by taking $\lambda\rightarrow0$.

A direct consequence of (\ref{verifying conditions}) is the condition
(\ref{upper bound criterion in exponential approximations}) in Theorem
\ref{exponential approximation}. Moreover, if we let $C=\overline{B_{\boldsymbol{w},\lambda}}$
in (\ref{verifying conditions}), by taking $\lambda\rightarrow0$
we easily obtain that $I(\boldsymbol{w})\leqslant\overline{I}(\boldsymbol{w})$,
where $\overline{I}$ is the function given by (\ref{rate function given in exponential approximation}). 

It remains to show that $\overline{I}(\boldsymbol{w})\leqslant I(\boldsymbol{w})$,
and we only consider the nontrivial case $I(\boldsymbol{w})=\alpha<\infty.$
It follows that $I(\boldsymbol{w})=\Lambda(w),$ where $w\in\mathcal{H}\subset\mathcal{A}_{p}$
with $F(w)=\boldsymbol{w}.$ Let $\boldsymbol{w}_{m}=\Psi_{m}(w).$
By Lemma \ref{Cameron-Martin convergence} we know that $\boldsymbol{w}_{m}\rightarrow\boldsymbol{w}$
under $d_{p}.$ Therefore, for any fixed $\lambda>0$, 
\[
\inf_{\boldsymbol{w}'\in B_{\boldsymbol{w},\lambda}}I'_{m}(\boldsymbol{w}')\leqslant I'_{m}(\boldsymbol{w}_{m})\leqslant\Lambda(w)=I(\boldsymbol{w})
\]
when $m$ is large. By taking ``$\liminf_{m\rightarrow\infty}$''
and ``$\sup_{\lambda>0}$'', we obtain that $\overline{I}(\boldsymbol{w})\leqslant I(\boldsymbol{w}).$

Now the proof is complete.

\end{proof}

\begin{rem}
In some literature (in particular, in \cite{Yoshida}), the Sobolev
norms over $(W,\mathcal{H},\mathbb{P})$ are defined in terms of the
Ornstein-Uhlenbeck operator, which can be regarded as the infinite
dimensional Laplacian under the Gaussian measure $\mathbb{P}.$ An
advantage of using such norms is that they can be easily extended
to the fractional case. According to the well known Meyer's inequalities,
such norms are equivalent to the ones we have used here which are
defined in terms of the Malliavan derivatives. Therefore, the LDP
for the corresponding capacities under these Sobolev norms holds in
exactly the same way.
\end{rem}

\section*{Acknowledgements}

The research of the authors is supported by the Oxford-Man Institute
in University of Oxford, and the first two authors are also supported
by ERC (Grant Agreement No.291244 Esig). We would like to thank Professor
Inahama for bringing the papers \cite{Aida}, \cite{Higuchi}, \cite{Inahama1}
and \cite{Inahama2} to our attention.

\end{document}